\documentclass[11pt]{amsart}
\usepackage{cite}
\usepackage{geometry}
\usepackage{amsmath,amssymb,graphicx,blkarray,amsthm} 
\usepackage{bbm}
\usepackage{color,subfig}
\usepackage[table]{xcolor}
\usepackage[colorlinks,linkcolor=blue,citecolor=blue,urlcolor=black,bookmarks=false,hypertexnames=true]{hyperref} 
\usepackage{float}

\usepackage{tikz}
\usetikzlibrary{calc,angles,positioning,quotes,decorations}
\usepackage{tkz-euclide}

\usepackage{verbatim}
\usepackage{tikz-cd}

\newenvironment{psmallmatrix}
  {\left(\begin{smallmatrix}}
  {\end{smallmatrix}\right)}

\theoremstyle{plain}
\newtheorem{theorem}{Theorem}
\newtheorem{proposition}[theorem]{Proposition}
\newtheorem{lemma}[theorem]{Lemma} 
\newtheorem{corollary}[theorem]{Corollary}
\newtheorem{fact}[theorem]{Fact}

\theoremstyle{definition}
\newtheorem{definition}[theorem]{Definition}
\newtheorem{remark}[theorem]{Remark}
\newtheorem{question}[theorem]{Question}
\newtheorem{example}[theorem]{Example}

\newcommand{\dd}{{\mathrm{d}}}

\newcommand{\exend}{\hfill $\Diamond$}

\DeclareMathOperator{\norm}{norm}
\DeclareMathOperator{\Aut}{Aut}

\begin{document}

\title{Admissible Reversing and Extended Symmetries for Bijective Substitutions}

\subjclass[2010]{52C23, 37B10, 37B52, 20B27}
\keywords{Extended symmetries, automorphism groups, substitution subshifts, aperiodic tilings}

\author{\'{A}lvaro Bustos}
\address{Departamento de Ingeniería Matemática,
Universidad de Chile, 
 \newline
\hspace*{\parindent}Beauchef 851, Santiago, Chile}
\email{abustos@dim.uchile.cl}

\author{Daniel Luz}
\author{Neil Ma\~nibo} 
\address{Fakult\"at f\"ur Mathematik, Universit\"at Bielefeld, \newline
\hspace*{\parindent}Postfach 100131, 33501 Bielefeld, Germany}
\email{\{dluz,cmanibo\}@math.uni-bielefeld.de }

\begin{abstract}
In this paper, we deal with reversing and extended symmetries of shifts generated by bijective substitutions. We provide equivalent conditions for a permutation on the alphabet to generate a reversing/extended symmetry, and algorithms how to check them. Moreover, we show that, for any finite group $G$ and any subgroup $P$ of the $d$-dimensional hyperoctahedral group, there is a bijective substitution which generates an aperiodic hull with symmetry group $\mathbb{Z}^{d}\times G$ and extended symmetry group $(\mathbb{Z}^{d} \rtimes P)\times G$.  
\end{abstract}

\maketitle

\section{Introduction}
The study of symmetry groups, often also known as automorphism groups, is an important part of the analysis of a dynamical system, as it can offer insight on the behaviour of the system, as well as allowing classifications of distinct families of dynamical systems (acting as a conjugacy invariant). In particular, symmetry groups of  shift spaces have been thoroughly studied (see e.g. the analysis of the symmetry group of the full shift \cite{BLR}, the series of works on symmetries in low-complexity subshifts \cite{CK, CQY,DDMP}, and recent works on shifts of algebraic and number-theoretic origin \cite{BBHLN, FY}).

Symmetries of subshifts can be algebraically defined as elements of  the topological centraliser of the group $\left\langle\sigma\right\rangle$ generated by the shift, seen as a  subgroup of the space $\mathrm{Aut}(\mathbb{X})$ of all self-homeomorphisms of $\mathbb{X}$ onto itself.
Thus, a natural question at this point is whether the corresponding normaliser has an interesting dynamical interpretation as well. This leads to the concept of \emph{reversing symmetries} (for $d=1$); see \cite{BR,Goodson,BRY}, the monograph \cite{OS} for a group-theoretic exposition, and \cite{LR} for a more physical background. These are special types of flip conjugacies; see \cite{BM}. In higher dimensions, one talks of  \emph{extended symmetries}; see \cite{Baake2,BRY}, which 
are examples of $\text{GL}(d,\mathbb{Z})$-conjugacies; compare \cite{Lab,BBHLN}.  
 These kinds of maps are related to phenomena such as palindromicity and several properties of
 geometric and topological nature, which is more evident in the higher-dimensional setting \cite{BRY,B}.

High complexity is often (but not always, see for instance the square-free subshift \cite{BBHLN}) linked to a complicated symmetry group. For instance, determining whether the symmetry groups of the full shifts in two and three symbols are isomorphic has consistently proven to be a difficult question \cite{BLR}. The low-complexity situation, thus, often allows for a more in-depth analysis and more complete descriptions, up to and including explicit computation of these groups in many cases.

The particular case of substitutive subshifts 
has gathered significant attention and here a lot of progress has been made; see \cite{MY,KY}.
Unsurprisingly, the presence of non-trivial symmetries is also tied to the spectral structure of the underlying dynamical system; see \cite{Q,Frank2}. 
In this work, we restrict to systems generated by bijective substitutions, both in one and in higher dimensions.  These substitutions are typically $n$-to-$1$ extensions of odometers and generate coloured tilings of $\mathbb{Z}^d$ by unit cubes, where one usually identifies a letter with a unique colour; see \cite{Frank2}. We compile and extend  known properties about this family of substitutive subshifts regarding symmetries. Some natural questions in this direction are:
\begin{enumerate}
    \item What kinds of groups can appear as symmetry groups/extended symmetry groups of specific substitutive subshifts?
    \item Given a specific group $G$, can we construct a substitution whose associated subshift has $G$ as its symmetry group/extended symmetry group? 
\end{enumerate}

Both questions are accessible for bijective substitutions. For symmetry groups, the second question is answered in full in \cite{DDMP}, which extends to higher dimensions with no additional assumptions because the result does not depend on the geometry of the substitution; see \cite{CP} for realisation results for more general group actions. We add to such known results in Theorem~\ref{thm:alphabet size limitation}. 
Aperiodicity also plays a key role here, which can easily be confirmed in the bijective setting; see Propositions~\ref{prop: aperiodicity proximal pair} and \ref{prop: aperiodicity HD proximal}. 

On the other hand, the existence of non-trivial reversing or extended symmetries depends heavily on the geometry and requires more in terms of the relative positions of the permutations in the corresponding supertiles, the expansive maps, and the shape of the supertiles themselves.
In Theorem~\ref{thm: reversing symmetries iff}, we provide
equivalent conditions for the existence of non-trivial reversing symmetries, which we generalise to higher dimensions in Theorem~\ref{thm: extended symmetries iff} to cover extended symmetries. 

As a corollary, in any dimension $d$, given a finite group $G$ and a subgroup $P$ of the hyperoctahedral group $P$, we provide a construction in Theorem~\ref{thm: G and P construction} of a bijective substitution whose underlying shift space has symmetry group and extended symmetry group $\mathbb{Z}^{d}\times G$ and $(\mathbb{Z}^{d}\rtimes P) \times G$, respectively. A similar construction with a different structure of the extended symmetry group is done in Theorem~\ref{thm: G and P non id}. We also provide algorithms on how one can check whether there exist non-trivial symmetries and extended symmetries for a given substitution $\varrho$; see Sections~\ref{sec: symmetries} and \ref{sec: reversing symm}.

\section{Bijective constant-length substitutions}

\subsection{Setting and basic properties}

Let $\mathcal{A}$ be a finite alphabet and $\mathcal{A}^{+}=\bigcup_{L\geqslant 1}\mathcal{A}^{L}$ be the set of finite non-empty words over $\mathcal{A}$; we shall write $\mathcal{A}^* = \mathcal{A}^+\cup\{\varepsilon\}$, where the latter is the empty word. A \emph{substitution} is a map $\varrho\colon \mathcal{A}\to\mathcal{A}^{+}$. If there exists an $L\in\mathbb{N}$ such that $\varrho(a)\in \mathcal{A}^{L}$ for all $a\in\mathcal{A}$, $\varrho$ is called a \emph{constant-length} substitution.  If there exists a power $k$ such that $\varrho^{k}(a)$ contains all letters in $\mathcal{A}$, for every $a\in\mathcal{A}$, we call $\varrho$ \emph{primitive}.

The \emph{full shift} is the set $\mathcal{A}^{\mathbb{Z}}$ of all functions (configurations) $x\colon \mathbb{Z}\to\mathcal{A}$. More generally, we define the $d$-dimensional full shift as the set $\mathcal{A}^{\mathbb{Z}^d}$. To this space, we assign the product topology, giving $\mathcal{A}$ the discrete topology. This is a particular version of the \emph{local topology} used in tiling spaces and discrete point sets, in which two tilings (or point sets) $x$ and $y$ are said to be $\varepsilon$-\emph{close} if a small translation of $x$ (of magnitude less than $\varepsilon$) matches $y$ on a large ball (of radius at least $1/\varepsilon$) around the origin; this can be used to define a metric $\dd(x,y)$. In the particular case of shift spaces seen as tiling spaces, since tiles are aligned with $\mathbb{Z}^d$, we can disregard the translation and get, e.g., the following as an equivalent metric:    
\[\dd(x,y) = 2^{-\inf\left\{n\::\:x\rvert_{[-n,n]^d} \neq y\rvert_{[-n,n]^d} \right\}}.
\]
This space is endowed with the \emph{shift action} of $\mathbb{Z}^d$ on $\mathcal{A}^{\mathbb{Z}^d}$, which is the action of $\mathbb{Z}^d$ over configurations by translation, and can be defined via the equality $(\sigma_{\boldsymbol{n}}(x))_{\boldsymbol{m}} = x_{\boldsymbol{n}+\boldsymbol{m}}$ for all $x\in\mathcal{A}^{\mathbb{Z}^d}$, $\boldsymbol{m},\boldsymbol{n}\in\mathbb{Z}^d$ (in particular, in one dimension we have the shift map $\sigma=\sigma_1$, which completely determines the group action).

A \emph{subshift} is a topologically closed subset $\mathbb{X}\subseteq\mathcal{A}^{\mathbb{Z}^d}$ which is also invariant under the shift action. Thus, a subshift combined with the restriction of this group action to $\mathbb{X}$ defines a topological dynamical system, which can be endowed with one or more measures to obtain a measurable dynamical system. In the one-dimensional case, the \emph{language} (or \emph{dictionary}) of a subshift $\mathbb{X}$ is the set of all words that may appear contained in some $x\in\mathbb{X}$, that is:
\[\mathcal{L}(\mathbb{X}) = \{x|_{[0,n]}\::\:x\in\mathbb{X}, n\geqslant 0 \}\cup \left\{\varepsilon\right\}.\]

We may verify that any nonempty set of words $\mathcal{L}$ which is \emph{extensible} (that is, any $w\in\mathcal{L}$ is a subword of a longer word $w'\in\mathcal{L}$) and closed under taking subwords is the language of a subshift, and two subshifts are equal if and only if they share the same language.

Higher-dimensional subshifts have a similar combinatorial characterisation, where the role of words is taken by \emph{patterns}, finite configurations of the form $P\colon U\subset\mathbb{Z}^d\to\mathcal{A},\lvert U\rvert < \infty$; we identify a pattern with any of its translations. In most cases\footnote{Pattern shapes do matter when studying certain generalisations of topological mixing in the $d$-dimensional setting, where either restricting ourselves to specific shapes (rectangles, L-shapes, hollow rectangles, etc.) or allowing arbitrary ones may be preferrable depending on context. However, we are not concerned with these kinds of properties here.} (and, in particular, in the rest of this work), it makes no difference to allow arbitrary ``shapes'' $U$ or to restrict ourselves to only rectangular patterns, i.e. products of intervals of the form $U=\prod_{i=1}^d[0,n_i-1]$. Regardless of our chosen convention, we collect all valid patterns $x\rvert_U$ that appear in some $x\in\mathbb{X}$ into a set $\mathcal{L}(\mathbb{X})$ as above, which we once again call the \emph{language} of $\mathbb{X}$. As in the one-dimensional case, a language closed under taking subpatterns and  where every pattern of shape $U$ is contained in a pattern of shape $V\supset U$ for any larger (finite) $V$ defines a unique subshift, and vice versa. 

Thus, given that iterating a primitive substitution $\varrho\colon\mathcal{A}\to\mathcal{A}^L$ of constant length $L>1$ over a symbol $a\in\mathcal{A}$ produces words of increasing length, the set $\mathcal{L}_\varrho$ of all words that are subwords of some $\varrho^k(a)$ for some $k\geqslant 1$ and $a\in\mathcal{A}$ is the language of a unique subshift that depends only on $\varrho$, which we shall call the \emph{substitutive subshift} defined by $\varrho$ and denote by $\mathbb{X}_\varrho$. This definition extends to $d$-dimensional \emph{rectangular substitutions} $\varrho\colon\mathcal{A}\to\mathcal{A}^R$ (where $R$ is a product of intervals), which are higher-dimensional analogues of constant-length substitutions; see \cite{Frank2,Q,Bartlett}.
It is well known that the primitivity of $\varrho$ implies that $\mathbb{X}_{\varrho}$ is strictly ergodic (uniquely ergodic and minimal); see \cite{Q,BG}. 
We refer the reader to \cite{MR} for a treatment of substitutions which are non-primitive.

\begin{definition}
    A constant-length substitution $\varrho\colon\mathcal{A}\to\mathcal{A}^L$ is called \emph{bijective} if the map which is given by $\varrho_j\colon a\mapsto\varrho(a)^{ }_j$ is a bijection on  $\mathcal{A}$, for all indices $0 \leqslant j \leqslant L-1$. Equivalently, $\varrho$ is bijective if there exist $L$ (not necessarily distinct) bijections $\varrho^{ }_0,\dotsc,\varrho^{ }_{L-1}\colon\mathcal{A}\to\mathcal{A}$ such that $\varrho(a) = \varrho^{ }_0(a)\dotsc\varrho^{ }_{L-1}(a)$ for every $a\in\mathcal{A}$. We shall refer to the mapping $\varrho^{ }_j$ as the $j$-th \emph{column} of the substitution $\varrho$.
\end{definition}

Consider $\big\{\varrho^{ }_{j}\big\}_{j=0}^{L-1}\subset S_{|\mathcal{A}|}$. Let $\Phi\colon S_{|\mathcal{A}|}\to \text{GL}(|\mathcal{A}|,\mathbb{Z})$ be the representation via permutation matrices. One then has the following; compare \cite[Cor.~1.2]{Frank2}.

\begin{fact} \label{fact:frequency_bij}
Let $\varrho$ be a primitive, bijective substitution, whose columns are given by 
$\left\{\varrho^{ }_0,\ldots,\varrho^{ }_{L-1}\right\}$. Then the substitution matrix $M$ is given by $M=\sum_{j=0}^{L-1}\Phi(\varrho^{-1}_{j})$. Moreover, $(1,1,\ldots,1)^{T}$ is a right Perron--Frobenius eigenvector of $M$, so each letter has the same frequency for every element in the hull $\mathbb{X}_\varrho$, i.e., $\nu_a=\frac{1}{|\mathcal{A}|}$ for all $a\in\mathcal{A}$ and all $x\in\mathbb{X}_{\varrho}$. \qed
\end{fact} 

Define the \emph{$n$-th column group} $G^{(n)}$ to be the following subgroup of the symmetric group of bijections $\mathcal{A}\to\mathcal{A}$:
        \[G^{(n)} := \langle \{\varrho^{ }_{j_1}\circ\cdots\circ\varrho^{ }_{j_n}\:|\: 0\leqslant j_1,\dotsc,j_n \leqslant L-1\}\rangle. \]

As it turns out, the groups $G^{(n)}$ generated by the columns  give a good description of the substitution $\varrho$ in the bijective case; see
\cite{KY} for its relation to the corresponding Ellis semigroup of $\mathbb{X}_\varrho$. 
The primitivity of $\varrho$ may be characterised entirely by this family of groups, as seen below. Recall that a subgroup $G\leqslant S_n$ of the symmetric group on $\{1,\dotsc,n\}$ is \emph{transitive} if for all $1\leqslant j,k\leqslant n$ there exists $\tau\in G$ such that $\tau(j)=k$. Here, we let $N\in\mathbb{N}$ be the minimal power such that $\varrho^{N}_j=\text{id}$ for some $0\leqslant j\leqslant L^N-1$; compare \cite[Lem.~8.1]{Q}. In \cite{KY}, $G^{(N)}$ is called the \emph{structure group} of $\varrho$. 

\begin{proposition}
    \label{prop:group_primitivity}
    Let $\varrho\colon\mathcal{A}\to\mathcal{A}^L$ be a bijective substitution. Then, the following are equivalent:
    \begin{enumerate}
        \item The substitution $\varrho$ is primitive.
        \item All groups $G^{(n)},n\in\mathbb{N}$, are transitive.
        \item The group $G^{(N)}$ is transitive.
    \end{enumerate}
\end{proposition}

\begin{proof}
    Evidently, $(2)\implies(3)$, so we only need to prove $(3)\implies(1)\implies(2)$. To see the first implication, note first that the columns of the iterated substitution $\varrho^N$ are compositions of the form $\varrho^{ }_{j_1,\dotsc,j_N}:=\varrho^{ }_{j_1}\circ\cdots\circ\varrho^{ }_{j_N},0 \leqslant j_1,\dotsc,j_N \leqslant L-1$, that is, for any $a\in\mathcal{A}$ the following holds:
        \[\varrho^N(a) = \varrho^{ }_{0,\dotsc,0,0}(a)\varrho^{  }_{0,\dotsc,0,1}(a)\dotsc \varrho^{ }_{0,\dotsc,0,L-1}(a)\varrho^{ }_{0,\dotsc,1,0}(a)\dotsc\varrho^{ }_{L-1,\dotsc,L-1,L-1}(a).\]
    Since, by $(3)$, the group $G^{(N)}$ is transitive, the substitution matrix $M_{\varrho^N}$ is irreducible, i.e. it is the adjacency matrix of a strongly connected digraph. In other words, for all $a,b\in\mathcal{A}$, there exists a composition of columns $q,q',\dotsc,q''$ of $\varrho^N$ such that $q\circ q'\circ\cdots\circ q''(a) = b$, which may be identified with a path in the graph whose vertices are the letters of $\mathcal{A}$ and with one edge from $c$ to $r(c)$ for any $c\in\mathcal{A}$ and column $r$. The choice of $N$ also shows that $M_{\varrho^N}$ has a non-zero diagonal, since one of the columns of $\varrho^N$ is the identity. These two conditions immediately imply that $M_{\varrho^N}$ is a primitive matrix (see \cite[Ch.~2]{LMa}) which in turn implies primitivity of $\varrho$, as desired.
    
    To prove $(1)\implies(2)$, note that primitivity of $\varrho$ implies that, for some $k>0$ and for all $a\in\mathcal{A}$, the word $\varrho^k(a)$ contains all symbols of the alphabet $\mathcal{A}$, including $a$ itself. Since the columns of $\varrho^k$ generate $G^{(k)}$, this implies that for all $a,b\in\mathcal{A}$ there is some generator of this group that maps $a$ to $b$, i.e. $G^{(k)}$ is transitive. Since $\varrho^k(a)$ contains $a$ as a subword, this implies that $\varrho^{2k}(a)$ contains $\varrho^k(a)$ as a subword, and, by induction, that $\varrho^{mk}(a)$ contains $\varrho^k(a)$ as a subword for all $m\geqslant 1$; thus, all groups $G^{(mk)}$ are transitive. Now, it is easy to see that $G^{(n)}\leqslant G^{(d)}$ if $d\mid n$. Then, for all $n\in\mathbb{N}$, $G^{(n)}$ has $G^{(nk)}$ as a transitive subgroup and hence it is transitive.
\end{proof}

The bijective structure of $\varrho$ can also be exploited to conclude the aperiodicity of $\mathbb{X}_{\varrho}$ by just looking at simple  features of $\varrho$. Below, we provide several criteria for aperiodicity in terms of $|\mathcal{A}|$, $L$, and the existence of certain legal words.

\begin{proposition} \label{prop:periodic_conditions}
Let $\mathbb{X}_\varrho$ be the hull of a primitive, bijective substitution $\varrho$ of length $L$ on a finite alphabet $\mathcal{A}$. If $\mathbb{X}_\varrho$ is periodic with least period $p$, then it has to satisfy the following conditions:
\begin{enumerate}
    \item $|\mathcal{A}|$ divides $p$
    \item $L$ does not divide $|\mathcal{A}|$.     
\end{enumerate}
\end{proposition}
\begin{proof}
From Fact~\ref{fact:frequency_bij} we know that every letter has the same frequency. An element of a periodic shift is just a concatenation of its periods and thus every letter has the same frequency in every period. This is only possible if every letter appears equally often in the period and thus the period length has to be a multiple of the alphabet size, which settles the first claim.

Let us assume that $w^\infty$ is a periodic word with $p$ as least period. Without loss of generality, choose a power $\varrho^k$ such that its first column is the identity, and so $w^\infty$ is fixed by $\varrho^k$. We choose $c$ and $d$ minimal such that:
\begin{align}\label{eq:length}
c L= d p \quad \Longleftrightarrow \quad c b |\mathcal{A}| = d a |\mathcal{A}|.    
\end{align}

We apply $\varrho^{-1}$ to $w^\infty|_{\left[0,cL-1\right]}$ of length $cL$, which has a unique pre-image in $\mathcal{L}$ since $\varrho$ is injective on letters, i.e., $\varrho(a)=\varrho(b)$ if and only if $a=b$. Then this segment must be of the form $x_1, \cdots x_c$. Applied to $w^\infty$, it yields $ \cdots x_c x_1, \cdots x_c x_1 \cdots $, which means that $w^\infty$ is $c$-periodic. 
Since $p$ is the least period $c=e p$ but then $c=e a| \mathcal{A}| $, making the factor $| \mathcal{A}|$ in Eq.~\eqref{eq:length} redundant, and thus contradicting the minimality of $c$. 
\end{proof}

Another way to get aperiodicity is through the existence of proximal pairs; see \cite[Sec.~3.2.1]{DDMP} and \cite[Cor.~4.2 and Thm.~5.1]{BG}. Two elements $x\neq y\in (\mathbb{X},\sigma)$ are said to be proximal if there exists a subsequence $\left\{n_k\right\}$ of $\mathbb{N}$ or $-\mathbb{N}$ such that $\dd(\sigma^{n_k}x,\sigma^{n_k}y)\to 0$ as $k\to\infty$. A stronger notion is that of asymptoticity, which requires 
$\dd(\sigma^{n}x,\sigma^{n}y)\to 0$ as $n\to\infty$ or $-\infty$. For bijective substitutions, these two notions are equivalent, and asymptotic pairs
are completely characterised by fixed points of $\varrho$; see \cite{KY}.

Consider a one-dimensional substitution $\varrho$ and a fixed point $w$ arising from a legal \emph{seed} $a|b$, i.e., $w=\varrho^{\infty}(a|b)$.  Here, the vertical bar represents the location of the origin, and the letter $a$ generates all the letters at the negative positions, while $b$ does the same for all non-negative ones. 
Two fixed points $w_1,w_2\in \mathbb{X}_{\varrho}$ generated by $a_1|b_1$ and $a_2|b_2$ are \emph{left-asymptotic} if they agree at all negative positions and disagree for all non-negative positions. Right-asymptotic pairs are defined in a similar manner. 
We have the following equivalent condition for aperiodicity in terms of existence of certain legal words; compare \cite[Prop.~4.1]{KY}

\begin{proposition}\label{prop: aperiodicity proximal pair}
Let $\varrho$ be a primitive, bijective substitution on a finite alphabet $\mathcal{A}$ in one dimension. Then the hull $\mathbb{X}_{\varrho}$ is aperiodic if and only if there exist distinct legal words of length \textnormal{2} which either share the same starting or ending letter.  
\end{proposition}

\begin{proof}
Let $\varrho:=\varrho^{ }_{0}\cdots\varrho^{ }_{L-1}$, with $\varrho_i\in G$. 
Choosing $k=\text{lcm}(|\varrho^{ }_0|,|\varrho^{ }_{L-1}|)$, we get that the first and the last columns of $\varrho^{k}$ are both the identity, i.e., $\varrho_0^{k}(a)=\varrho_{L^k-1}^{k}(a)=a$ for all $a\in\mathcal{A}$. If there exist $ab,ac\in\mathcal{L}_{\varrho}$ with $b\neq c$, the bi-infinite fixed points $\varrho^{\infty}(a|b)$ and $\varrho^{\infty}(a|c)$ they generate under $\varrho^{k}$ coincide in all negative positions and differ in at least one non-negative position, and  hence are left-asymptotic and proximal. Since $\mathbb{X}_{\varrho}$ is minimal and admits a proximal pair, all of its elements must then be aperiodic. 
Now suppose that every letter has a unique predecessor and successor in $\mathcal{A}$. This means that every element $x\in \mathbb{X}_{\varrho}$ is uniquely determined by the letter at the origin. From the finiteness of $\mathcal{A}$, one gets $x=w^{\infty}$ and hence is periodic, from which the periodicity of the hull follows.   
\end{proof}

\begin{example}
The substitution $\varrho\colon a\mapsto aba,b\mapsto bab$ is primitive, bijective and admits a periodic hull. Here, the only legal words of length $2$ are $ab$ and $ba$. Note that $\varrho$ is of height $2$ and generates the same hull as the substitution $\varrho^{\prime}\colon a,b\mapsto ab$.  \exend
\end{example}

\subsection{Symmetries}\label{sec: symmetries}

In the following sections, we deal with the \emph{symmetry groups} of our subshifts of interest, which are certain homeomorphisms of the shift space which preserve the dynamics of the shift action in a specific sense.

\begin{definition}
Let $\mathbb{X}$ be a $\mathbb{Z}^d$-subshift. The \emph{symmetry group} (often called \emph{automorphism group}\footnote{In this work, we follow the notational conventions of \cite{BRY}, and thus we avoid the term ``automorphism group'' as it may be understood as the set of all homeomorphisms $f\colon\mathbb{X}\to\mathbb{X}$.}) is the set $\mathcal{S}(\mathbb{X})$ of all homeomorphisms $\mathbb{X}\to\mathbb{X}$ which commute with the shift action, i.e.,     \begin{equation}\label{eq:dfn_of_symmetry}
    (\forall\boldsymbol{n}\in\mathbb{Z}^d)\colon \sigma_{\boldsymbol{n}}\circ f = f\circ\sigma_{\boldsymbol{n}}.
\end{equation}
\end{definition}

That is, $\mathcal{S}(\mathbb{X})$ is the centraliser of the set of shift maps in the group of all self-homeomorphisms of the space $\mathbb{X}$. In this context, every symmetry $f\in\mathcal{S}(\mathbb{X})$ is entirely determined by its \emph{local function}, which is a mapping $F\colon\mathcal{A}^U\to\mathcal{A}$, with $U\subset\mathbb{Z}^d$ finite, such that for every $\boldsymbol{n}\in\mathbb{Z}^d$, $f(x)_{\boldsymbol{n}} = F(x\rvert_{\boldsymbol{n}+U})$. This fact is known as the \emph{Curtis--Hedlund--Lyndon (or CHL) theorem}; see \cite{LMa}. We say that $f$ has \emph{radius} $r\geqslant 0$ if this is the least non-negative integer such that $U\subseteq[-r,r]^d$.

Symmetry groups of one-dimensional bijective substitutions are a thoroughly studied subject, both in the topological and ergodic-theoretical contexts. Complete characterisations of these groups are known, as seen in e.g. \cite{C} for a two-symbol alphabet, or \cite{LMe} for a characterisation in the measurable case; see also \cite{Frank2,CQY} for further elaboration in the description of the symmetries in this category of subshifts. The following theorem summarizes this classification:

\begin{theorem}
    \label{thm:symm_group}
    Let $\mathbb{X}_\varrho$ be the hull generated by an aperiodic, primitive, bijective substitution $\varrho$ on $\mathbb{Z}^d$. Then, the symmetry group $\mathcal{S}(\mathbb{X}_\varrho)$ is isomorphic to the  direct product of  $\mathbb{Z}^d$, generated by the shift action, with a finite group of radius-$0$ sliding block codes $\tau_\infty\colon\mathbb{X}_\varrho\to\mathbb{X}_\varrho$ given by $\tau_{\infty}((x^{}_{\boldsymbol{j}})^{}_{\boldsymbol{j}\in\mathbb{Z}^d})=(\tau(x^{}_{\boldsymbol{j}}))^{}_{\boldsymbol{j}\in\mathbb{Z}^d}$ for some bijection $\tau\colon\mathcal{A}\to\mathcal{A}$.
    
    Furthermore, let $N$ be any integer such that $\varrho_{\boldsymbol{j}}^N$ is the identity for some $\boldsymbol{j}$ (note that such an $N$ always exists). Then, $\tau\colon\mathcal{A}\to\mathcal{A}$ induces a symmetry if and only if $\tau\in\mathrm{cent}_{S_{\lvert\mathcal{A}\rvert}}G^{(N)}$. \qed 
\end{theorem}

As a consequence, every symmetry on $\mathbb{X}_\varrho$ is a composition of a shift map and a radius-zero sliding block code as above. These conditions arise as a consequence of such a symmetry having to preserve the supertile structure of any $x\in\mathbb{X}_\varrho$ at every scale, which in particular implies that a level-$k$ supertile $\varrho^k(a),a\in\mathcal{A}$ has to be mapped to some $\varrho^k(b)$ for some other $b\in\mathcal{A}$ by the ``letter exchange map'' $\tau$. The choice of $N$ above ensures that, when $k$ is a multiple of $N$, the equality $a=b$ holds, which implies that $\tau$ commutes with the columns of $\varrho^N$, and thus $\varrho^N\circ\tau_\infty=\tau_\infty\circ\varrho^N$. This in turn implies Eq.~\eqref{eq:dfn_of_symmetry}. For further elaboration on the proof of the above result, the reader may consult \cite{Frank2,CQY}, among others.

\begin{example}
    \label{ex:cyclic_subst}
    Consider the following substitution $\varrho$ on the three-letter alphabet $\mathcal{A}=\{a,b,c\}$:
    \begin{align*}
        \varrho\colon a &\mapsto abc, \\
        b &\mapsto bca, \\
        c &\mapsto cab.
    \end{align*}
    The columns correspond to the three elements of the cyclic group generated by $\tau=(a\,b\,c)$. It is not hard to verify that the only elements of $S_3=D_3$ that commute with $\tau$ are the powers of $\tau$ themselves, and thus $\mathcal{S}(\mathbb{X}_\varrho)\simeq \mathbb{Z}\times C_3$, with the finite subgroup $C_3$ being generated by the symmetries induced by the powers of $\tau$. \exend
\end{example}

As it turns out, Theorem \ref{thm:symm_group} provides an algorithm to compute $\mathcal{S}(\mathbb{X}_\varrho)$ explicitly. To introduce this algorithm, let us recall some easily verifiable facts from group theory \cite[Ch.~1 and ~5]{Hall}:
\begin{fact} 
    \label{fact:centraliser_by_generators}
    Let $G$ be any group and $H=\langle S\rangle \leqslant G$ a subgroup generated by $S\subset G$. Then,
    \begin{equation*}\label{eq:cent_by_gen_formula}
    \pushQED{\qed}
        \mathrm{cent}_G(H)=\{c\in G\:|\:(\forall h\in H)\colon ch=hc\}=\bigcap_{s\in S} \mathrm{cent}_G(s). \qedhere
        \popQED
    \end{equation*} 
\end{fact}

\begin{fact}
    \label{fact:permutation_groups_rep}
    Any permutation decomposes uniquely (up to reordering) as a product of disjoint cycles. Conjugation by some $\tau\in S_n$ can be computed from this decomposition using the identity:
        \[\tau(a_1\,a_2\,\dotsc\,a_n)\tau^{-1} = (\tau(a_1)\,\tau(a_2)\,\dotsc\,\tau(a_n)). \]
    A permutation $\tau\in S_n$ belongs to $\mathrm{cent}_{S_n}(\pi)$ if and only if $\tau\pi\tau^{-1}=\pi$, and thus:
        \begin{align*}
            \pi &=  (a_1\,a_2\,\dotsc\,a_{k_1})(b_1\,b_2\,\dotsc\,b_{k_2})\cdots(c_1\,c_2\,\dotsc\,c_{k_r})\\
            &=  (\tau(a_1)\,\tau(a_2)\,\dotsc\,\tau(a_{k_1}))(\tau(b_1)\,\tau(b_2)\,\dotsc\,\tau(b_{k_2}))\cdots(\tau(c_1)\,\tau(c_2)\,\dotsc\,\tau(c_{k_r})).
        \end{align*}
    Hence, the uniqueness of this decomposition implies that every cycle in the second decomposition is equal to a cycle of the same length in the first one. \qed
\end{fact}

Thus, to compute the letter exchange maps that determine $\mathcal{S}(\mathbb{X}_\varrho)$, we need to find all permutations $\tau$ that preserve certain cycle decompositions. We obtain the following procedure:
\vspace{1em}
\hrule
\vspace{0.3em}
{\noindent\small\textbf{Algorithm.} Assuming that $\varrho$ is a primitive, bijective, aperiodic substitution, the following algorithm computes $\mathcal{S}(\mathbb{X}^{ }_\varrho)$ explicitly.
\vspace{1em}
\begin{itemize}
    \item \textbf{Input:} $\varrho$ is a length-$L$ bijective substitution, which may be represented as a function (dictionary) $\varrho\colon\mathcal{A}\to\mathcal{A}^L$ or a set of $L$ permutations $\varrho_0,\varrho_1,\dotsc,\varrho_{L-1}\colon\mathcal{A}\to\mathcal{A}$, corresponding to each column.
    \item \textbf{Output:} A (finite) set of permutations $C$ forming a group, so that $\mathcal{S}(\mathbb{X}_\varrho)=\mathbb{Z}^d\times C$.
\end{itemize}
\vspace{1em}
\begin{itemize}
    \item[(1)] Compute the least positive integer $N$ such that $\varrho_{\boldsymbol{j}}^N$ is the identity on $\mathcal{A}$ for some column of the substitution $\varrho$. $N$ equals the least common multiple of all cycle lengths in the decomposition of the columns $\varrho_{\boldsymbol{j}}$ into disjoint cycles (and is thus finite).
    \item[(2)] Determine all columns $\varrho^{ }_{\boldsymbol{j}_1}\circ\cdots\circ\varrho^{ }_{\boldsymbol{j}_N}$ of the iterated substitution $\varrho^N$. This is a generating set for the group $G^{(N)}$.
    \item[(3)] For every column computed in (2), compute $G_{\boldsymbol{j}_1,\dotsc,\boldsymbol{j}_N}=\mathrm{cent}_{S_n}(\varrho^{ }_{\boldsymbol{j}_1}\circ\cdots\circ\varrho^{ }_{\boldsymbol{j}_N})$ by taking the cycle decomposition of this permutation (in where we identify $\mathcal{A}$ with the set $\{1,2,\dotsc,\left|\mathcal{A}\right|\}$) and employing the characterisation above.
    \item[(4)] Let $C = \bigcap_{\boldsymbol{j}_1,\dotsc,\boldsymbol{j}_n}G_{\boldsymbol{j}_1,\dotsc,\boldsymbol{j}_N}$. As $C$ can be biunivocally identified with the set of valid letter exchange maps modulo a shift, return $\mathcal{S}(\mathbb{X}_\varrho) = \mathbb{Z}^d\times C$ as output.
\end{itemize}}
\hrule
\vspace{1em}

Example \ref{ex:cyclic_subst} above corresponds to a simple case in which $G^{(N)}=G^{(1)}$ is a cyclic group, and we derive an abelian subgroup of $S_3$ corresponding to the valid letter exchange maps. We can use the above procedure to construct examples with more complicated symmetry groups, see Example~\ref{ex:quaternions}. 

\begin{example}
    \label{ex:quaternions}
    We take as alphabet the \emph{quaternion group} $Q=\{e,i,j,k,\bar{e},\bar{\imath},\bar{\jmath},\bar{k}\}$ (see \cite{Hall} for the multiplication table and basic properties of this group, which is generated by the two elements $i$ and $j$). With this, we construct a length-$3$ bijective substitution defined by right multiplication, $x\mapsto (x\cdot i)(x\cdot j)(x\cdot k)$, given in full by:
    \begin{align*}
        e &\mapsto ijk, & \bar{e} &\mapsto \bar{\imath}\bar{\jmath}\bar{k}, \\
        i &\mapsto\bar{e}k\bar{\jmath},& \bar{\imath} &\mapsto e\bar{k}j, \\
        j &\mapsto\bar{k}\bar{e}i, &
        \bar{\jmath} &\mapsto ke\bar{\imath},\\
        k &\mapsto j\bar{\imath}\bar{e}, &
        \bar{k} &\mapsto \bar{\jmath}ie.
    \end{align*}
    By direct computation, $G^{(n)}=G^{(1)}\simeq Q$ for all $n$, making the substitution primitive (as $Q$ acts transitively on itself in an obvious way). The three columns which generate $G^{(1)}$ are:
        \begin{align*}
            R_i & := (e\,i\,\bar{e}\,\bar{\imath})(j\,\bar{k}\,\bar{\jmath}\,k), \\
            R_j & := (e\,j\,\bar{e}\,\bar{\jmath})(i\,k\,\bar{\imath}\,\bar{k}), \\
            R_k & := (e\,k\,\bar{e}\,\bar{k})(j\,i\,\bar{\jmath}\,\bar{\imath}).
        \end{align*}

    Thus, the substitution $\varrho^3$ has as columns $R_{xyz}(g)=g\cdot xyz$ with $x,y,z\in\{i,j,k\}$; in particular, since $jik=e$, $\varrho^3$ must have an identity column. Also, since $G^{(3)}=G^{(1)}$, this group is the right Cayley embedding of $Q$ into $S_8$. By applying the above algorithm, we obtain that the group of letter exchange maps is generated by the following two permutations:
        \begin{align*}
            \pi_0 & := (e\,i\,\bar{e}\,\bar{\imath})(j\,k\,\bar{j}\,\bar{k}), \\
            \pi_1 & := (e\,j\,\bar{e}\,\bar{\jmath})(i\,\bar{k}\,\bar{i}\,k).
        \end{align*}
    We can verify that these permutations generate the \emph{left} Cayley embedding of $Q$ into $S_8$. Alternatively, if we consider the transposition $\nu=(k\,\bar{k})$, we can use Fact \ref{fact:permutation_groups_rep} above to see that $\pi_0 = \nu R_i\nu^{-1}$ and $\pi_1 = \nu R_{j}\nu^{-1}$, which in turn implies that the group generated by $\pi_0$ and $\pi_1$ is conjugate to the group generated by $R_i$ and $R_j$, the latter being isomorphic to $Q$. This shows that $\mathcal{S}(\mathbb{X}_\varrho)\simeq\mathbb{Z}\times Q$. \exend
\end{example}

It is well known that symmetry groups of aperiodic minimal one-dimensional subshifts are virtually $\mathbb{Z}$. The following result gives a full converse for shifts generated by bijective substitutions. 

\begin{theorem}[{\cite[Thm.~3.6]{DDMP}}]\label{thm: DDMP-result}
For any finite group $G$, there exists an explicit primitive, bijective substitution $\varrho$, on an alphabet on $\left|G\right|$ letters, such that $\mathcal{S}(\mathbb{X}_\varrho)\simeq\mathbb{Z} \times G$. \qed
\end{theorem}

The proof, which may be consulted in \cite{DDMP}, follows a similar schema to the analysis done in Example \ref{ex:quaternions} above. In \cite[Sec.~4.1]{Frank2}, it was shown that the number of letters needed in Theorem~\ref{thm: DDMP-result} is actually a tight lower bound. Below, we actually prove something stronger.

\begin{theorem}\label{thm:alphabet size limitation}
    Let $\varrho$ be an aperiodic, primitive, bijective substitution on the alphabet $\mathcal{A}$. If $\mathcal{S}(\mathbb{X}_\varrho)\simeq\mathbb{Z}\times G$, then $G$ must act freely on $\mathcal{A}$, and the order of $G$ has to divide $\left|\mathcal{A}\right|$.
\end{theorem}

\begin{proof}
    As seen in \cite[Sec.~4.1]{Frank2}, if we replace $\varrho$ with a suitable power, we may ensure that the word $\varrho^q(a)$ starts with $a$ and contains every other symbol, for all $a\in\mathcal{A}$. Thus, for any $\pi\in S_n$, the equality $\pi(a)=b$ implies $\pi(\varrho^q(a))=\varrho^q(b)$, which in turn determines the images of every symbol in the alphabet; the bound $\lvert G\rvert\leqslant\lvert\mathcal{A}\rvert$ follows from here.

    Note as well that, since $\varrho$ is bijective, if $\pi(a)\ne a$, then $\pi(c)\ne c$ for every $c\in\mathcal{A}$ as the words $\varrho^q(a)$ and $\varrho^q(b)$ are either equal or differ at every position. This implies that if $\pi$ has any fixed point then it must be the identity, i.e. that, if we identify $G$ with the corresponding group of permutations over $\mathcal{A}$, the action of $G$ on the alphabet is free. Equivalently, the stabiliser $\mathrm{Stab}(c)$ of any $c\in\mathcal{A}$ is the trivial subgroup.
    
   The elements of $G$ commute with every column of $\varrho^q$. Due to primitivity, there always exists a column $\varrho^*=\varrho^{ }_{\boldsymbol{j}_1}\circ\cdots\circ\varrho^{ }_{\boldsymbol{j}_q}$ which maps this $a$ to any desired $c\in\mathcal{A}$. Since $\varrho^*$ commutes with every $\pi\in G$ (i.e., it is an equivariant bijection for the action of $G$ on $\mathcal{A}$), we have that $\mathrm{Orb}(c) = \varrho^*[\mathrm{Orb}(a)]$, i.e. the orbit of $c$ under $G$ is necessarily the image of the orbit of $a$ under $\varrho^*$. 
  
    Thus, every orbit is a set of the same cardinality. This means that $G$ induces a partition of $\mathcal{A}$ into disjoint orbits of the same cardinality $\ell$, which then must divide $\left|\mathcal{A}\right|$. By the freeness of the group action and the orbit-stabiliser theorem, we have $\lvert G\rvert = \lvert\mathrm{Orb}(a)\rvert\cdot\lvert\mathrm{Stab}(a)\rvert = \ell$, and thus $\lvert G\rvert$ divides $\lvert\mathcal{A}\rvert$.
\end{proof}

\begin{remark}
It follows from Theorem~\ref{thm:alphabet size limitation} that the substitution in Example~\ref{ex:quaternions} is a minimal one in the sense that for one to get a $Q$-extension in $\mathcal{S}(\mathbb{X}_{\varrho})$, one needs at least eight letters. \exend
\end{remark}

\begin{remark}\label{rem: direct multidimensional generalization}
    At no point in the proof of Theorem~\ref{thm: DDMP-result} found in \cite{DDMP} nor in Theorem~\ref{thm:alphabet size limitation} above the fact that the substitution was one-dimensional is actually used. Thus, since Theorem \ref{thm:symm_group} is known to be valid for general rectangular substitutions, the two theorems above must be valid in this more general setting as well, provided that the substitution is aperiodic in $\mathbb{Z}^d$, which one can always guarantee; see Propositions~\ref{prop: aperiodicity proximal pair} and ~\ref{prop: aperiodicity HD proximal}. 
    \exend
\end{remark}

\begin{corollary}
    For any finite group $G$, there exists an explicit primitive, bijective $d$-dimensional rectangular substitution $\varrho$, on an alphabet of $\left|G\right|$ letters, such that $\mathcal{S}(\mathbb{X}_\varrho)\simeq\mathbb{Z}^d\times G$. Furthermore, this is the least possible alphabet size: for any bijective, primitive and aperiodic $d$-dimensional rectangular substitution $\varrho$ on the alphabet $\mathcal{A}$, if $\mathcal{S}(\mathbb{X}_\varrho)\simeq\mathbb{Z}^d\times G$, then $G$ acts freely on $\mathcal{A}$, and $\left|G\right|$ divides $\left|\mathcal{A}\right|$. \qed
\end{corollary}

\section{Extended and reversing symmetries of substitution shifts}

\subsection{One-dimensional shifts}\label{sec: reversing symm}
Since the term symmetry group does not cover everything that can be thought of as a symmetry (in the geometric sense of the word) we introduce the notion of the reversing symmetry group; see \cite{BRY} for a detailed exposition. We will exclusively look at shift spaces $\mathbb{X}_\varrho$ which are given by a bijective, primitive substitution $\varrho$ and we will exploit this additional structure in determining the reversing symmetry group for this class.

\begin{definition}
The \emph{extended symmetry group} of a shift space $\mathbb{X}$ is given as 
\[
\mathcal{R}(\mathbb{X}):=\norm_{\Aut(\mathbb{X})}(\mathcal{G})= \{H \in \Aut(\mathbb{X}) \mid H \mathcal{G} = \mathcal{G} H \}
\]
where $\mathcal{G}$ is the group generated by the shift. 
In the case where the shift space is one-dimensional, we call $\mathcal{R}(\mathbb{X})$ the \emph{reversing symmetry group}  given by 
\[
\mathcal{R}(\mathbb{X})= \{H \in \Aut(\mathbb{X}) \mid H \circ \sigma\circ  H^{-1} = \sigma^{\pm 1} \}. 
\] 
\end{definition}

A homeomorphism $H\in \text{Aut}(\mathbb{X})$ which satisfies $H\circ\sigma \circ H^{-1}=\sigma^{-1}$ is called a \emph{reversor} or a \emph{reversing symmetry}.  
A Curtis--Hedlund--Lyndon-type characterisation of reversing symmetries, which incorporates the mirroring component ($\text{GL}(d,\mathbb{Z}$)-component in higher dimensions)
can be found in \cite{BRY}.

In what follows, we investigate the effect of a reversor $f$ on inflated words. Given a bijective substitution $
\varrho\colon \mathcal{A} \rightarrow \mathcal{A}^L, 
\varrho := \varrho^{ }_0\varrho^{ }_1 \cdots \varrho^{ }_{L-1}$, the mirroring operation $m$  acts on the columns of $\varrho$ via $
m(\varrho(a))=   \varrho^{ }_{L-1}(a)\cdots \varrho^{ }_2(a)  \varrho^{ }_0(a)$. We may extend this to infinite configurations over $\mathbb{Z}$ in two non-equivalent ways, given by $m(x)_k = x_{-k}$ and $m'(x)_k = x_{1-k}$, respectively; we shall refer to both as basic mirroring maps.

\begin{proposition}\label{prop: reversor form}
    Let $\varrho$ be an aperiodic, primitive, bijective substitution. Then, any reversor is a composition of a letter exchange map $\pi\in S_{n}$, where $n=|\mathcal{A}|$, a shift map $\sigma^k$ and one of the two basic mirroring maps $m$ or $m'$ (depending only on whether the substitution has odd or even length, respectively). 
\end{proposition}

See \cite[Prop.~1]{BRY} and Theorem~\ref{thm:symm_group}. This result, while desirable, is not immediately obvious (and can indeed be false for non-bijective substitutions, which may have reversors whose local functions have positive radius), and thus we show this result as a consequence of bijectivity.

\begin{proof}
    Suppose $f\colon\mathbb{X}_\varrho\to\mathbb{X}_\varrho$ is a reversor of positive radius $r\geqslant 1$, i.e.,  $x\rvert_{[-r,r]}=y\rvert_{[-r,r]}$ implies that one has $f(x)_0 = f(y)_0$. There is some power $k\geqslant 1$ such that the words $\varrho^k(a)$ of length $L^k$ are longer than the local window of $f$, which has length $2r+1$ (say, $k=\lceil\log(2r+1)/\log(L)\rceil$). Any point of $\mathbb{X}_\varrho$ is a concatenation of words of the form $\varrho^k(a),a\in\mathcal{A}$, which is unique up to a shift because of aperiodicity; see \cite{Sol}. 
    In particular, if we choose a fixed $x\in\mathbb{X}_\varrho$ and let $y = f(x)$, both points have such a decomposition. 

    Now, suppose that the value $L^k = 2\ell + 1$ is odd (the case where $L$ is even is dealt with similarly). By composing $f$ with an appropriate shift map (say $\tilde{f}=f\circ\sigma^h$), we can ensure that the central word $\varrho^k(a)$ in the aforementioned decomposition has support $[-\ell,\ell]$ for both $x$ and $y$ (note that we employ the uniqueness of the decomposition here, to avoid ambiguity in the chosen $h$). Since $L^k=2\ell+1\geqslant 2r+1$, we must have $\ell\geqslant r$, and thus $y_0$ is entirely determined by $x\rvert_{[-\ell,\ell]}$, which is a substitutive word $\varrho^k(a)$. But, since $\varrho$ is bijective, this word is in turn completely determined by its central symbol $x_0$.
    
    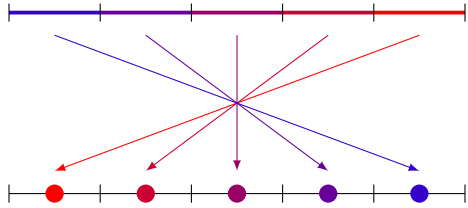
\begin{figure}[!h]
        \centering
        \begin{tikzpicture}[scale=1.2]
            \draw (0,0) -- (5,0);
            
            \draw[color=red!100!blue,line width=0.5mm] (5,2) -- (4,2);
            \draw[color=red!80!blue,line width=0.5mm] (4,2) -- (3,2);
            \draw[color=red!60!blue,line width=0.5mm] (3,2) -- (2,2);
            \draw[color=red!40!blue,line width=0.5mm] (2,2) -- (1,2);
            \draw[color=red!20!blue,line width=0.5mm] (1,2) -- (0,2);
            
            \draw[color=red!100!blue,-latex] (4.5,1.75) -- (0.5,0.25);
            \draw[color=red!80!blue,-latex] (3.5,1.75) -- (1.5,0.25);
            \draw[color=red!60!blue,-latex] (2.5,1.75) -- (2.5,0.25);
            \draw[color=red!40!blue,-latex] (1.5,1.75) -- (3.5,0.25);
            \draw[color=red!20!blue,-latex] (0.5,1.75) -- (4.5,0.25);
    
            \fill[color=red!100!blue] (0.5,0) circle (1mm);
            \fill[color=red!80!blue] (1.5,0) circle (1mm);
            \fill[color=red!60!blue] (2.5,0) circle (1mm);
            \fill[color=red!40!blue] (3.5,0) circle (1mm);
            \fill[color=red!20!blue] (4.5,0) circle (1mm);
                    
            \foreach \i in {0,1,...,5} {
                \draw (\i,-0.1) -- (\i,0.1);
                \draw[yshift=2cm] (\i,-0.1) -- (\i,0.1);    
            }
        \end{tikzpicture}
        \caption{A reversor $f$ establishes a 1-1 correspondence between words $\varrho^k(a)$ in a point $x$ and its image $f(x)$.}
        \label{fig:relabeling_argument}
    \end{figure}
    
    A similar argument shows that, for any $n\in\mathbb{Z}$, if $n\in mL^k + [-\ell,\ell]$, then $y_n$ depends only on the word $x\rvert_{-mL^k+[-\ell,\ell]}$, which contains (and is thus entirely determined by) $x_{-n}$. Since any point in $\mathbb{X}_\varrho$ is transitive, $\tilde{f}$ is entirely determined by the points $x$ and $y$, and thus, $\tilde{f}$ is a map of radius $0$. Equivalently, for some bijection $\pi\colon\mathcal{A}\to\mathcal{A}$, we have $\tilde{f}(x)_{-n} = \pi(x_n)$, that is, $\tilde{f} = f\circ\sigma_h = \pi\circ m$ (identifying $\pi$ with the letter exchange map $\mathcal{A}^{\mathbb{Z}}\to\mathcal{A}^{\mathbb{Z}}$). We conclude that $f$ is a composition of a letter exchange map, a mirroring map and a shift map.
\end{proof}

\begin{remark}
    \label{rem:higher dimensional ext symm decomposition}
    With some care, it can be shown that the same argument applies in the higher-dimensional case, where an element of the normaliser is a composition of a letter exchange map, a map of the form $f(x)_{\boldsymbol{n}} = x^{}_{A\boldsymbol{n}}$, with $A$ a linear map from the hyperoctahedral group (see Theorem \ref{thm: hyperoctahedral HD}, below), and a shift map; see \cite[Prop.~3]{BRY} for a more general formulation.\exend
\end{remark}

This result leads to the following criterion for the existence of a reversor in terms of the columns $\varrho^{ }_i$.

\begin{proposition} \label{prop:reversing necessary}
Let $\varrho$ be an aperiodic, primitive, bijective substitution $\varrho$ of length $L$ on a finite alphabet $\mathcal{A}$ of $n$ letters. Suppose that there exists a letter-exchange map $\pi\in S_n,\,\pi:\mathcal{A}\to\mathcal{A}$ which gives rise to a reversing symmetry. Then one has
\begin{equation}\label{eq: reversing symm condition}
\pi^{-1}\circ \varrho^{ }_i\circ \varrho^{-1}_{j}\circ \pi =\varrho^{ }_{L-(i+1)}\circ\varrho^{-1}_{L-(j+1)} 
\end{equation}
for all $0\leqslant i,j\leqslant L-1$, where $\varrho^{ }_i$  is the $i$-th column of $\varrho$ seen as an element of $S_n$. 
\end{proposition}
\begin{proof}

Let $a\in\mathcal{A}$. Let $m$ be the mirroring operation and suppose that there exists $\pi\in S_n$ such that $m\circ \pi$ extends to a reversor $ f\in\mathcal{R}(\mathbb{X}_\varrho)$. One then has 
\[
\varrho(a)=\varrho^{ }_0(a)\cdots \varrho^{ }_{L-1}(a)\,\overset{m}{\longmapsto}\, \varrho^{ }_{L-1}(a)\cdots \varrho^{ }_{0}(a)\, \overset{\pi}{\longmapsto}\, \pi\circ \varrho^{ }_{L-1}(a)\cdots \pi\circ\varrho^{ }_{0}(a).
\]
Since Proposition~\ref{prop: reversor form} guarantees that this must result to mapping substituted words to substituted words, one gets
\begin{equation}\label{eq: compare tau}
\pi\circ \varrho^{ }_{L-1}(a)\cdots \pi\circ\varrho^{ }_{0}(a)=\varrho^{ }_0(b)\cdots \varrho^{ }_{L-1}(b)=\varrho^{ }_0\circ\tau (a)\cdots \varrho^{ }_{L-1}\circ\tau(a), 
\end{equation}
where the permutation $\tau$ describes precisely this induced shuffling of inflation words. 
This yields
\[
\tau=\varrho^{-1}_j\circ\pi\circ \varrho^{ }_{L-(j+1)}
\]
for all $0\leqslant j\leqslant L-1$. Equating the corresponding right hand-sides for some pair $i,j$ yields Eq.~\eqref{eq: reversing symm condition}. The claim follows since this must hold for all $0\leqslant i,j\leqslant L-1$.
\end{proof}

\begin{theorem}\label{thm: reversing symmetries iff}
Let $\varrho$ be as in Proposition~\textnormal{\ref{prop:reversing necessary}}. Suppose further that  $\varrho^{ }_i=\varrho^{ }_{L-(i+1)}=\textnormal{id}$ for some $0\leqslant i \leqslant L-1$. Then, given a permutation (letter exchange map) $\pi\in S_n,\pi\colon \mathcal{A}\to \mathcal{A}$, the following are equivalent:

\begin{enumerate}
    \item[(i)] The letter exchange map $\pi$ gives rise to a reversing symmetry $f\in \mathcal{R}(\mathbb{X}_\varrho)\setminus \mathcal{S}(\mathbb{X}_\varrho)$ given by either $f(x)_n = \pi(x_{-n})$ or $f(x)_n = \pi(x_{1-n})$.
    \item[(ii)] The permutation $\pi$ satisfies the system of equations
    \begin{equation} \label{eq: reversing symm condition simple}
        \pi^{-1} \circ \varrho^{ }_i \circ \pi = \varrho^{ }_{L-(i+1)}     
    \end{equation}
    for all $0\leqslant i \leqslant L-1$. 
    \item[(iii)] There exist $\kappa^{ }_0,\kappa^{ }_1,\dotsc,\kappa^{ }_{L-1}\in S_n$, where each $\kappa^{ }_i$ satisfies $\kappa_i^{-1}\circ\varrho^{ }_i\circ\kappa^{ }_i = \varrho^{ }_{L-(i+1)}$, such that the following intersection of cosets is non-empty:
        \begin{equation}
            \label{eq: coset intersection}
            K = \bigcap_{i=0}^{L-1} \operatorname{cent}_{S_n}(\varrho^{ }_i)\kappa^{ }_i,
        \end{equation}
    and $\pi\in K$.
\end{enumerate}

\end{theorem}

\begin{proof}
It is clear that Eq.~\eqref{eq: reversing symm condition simple} implies Eq.~\eqref{eq: reversing symm condition}.
Note that it is sufficient to satisfy Eq.~\eqref{eq: reversing symm condition} for $j=i + 1$ mod $L$ as any term can be obtained by multiplying sufficient numbers of succeeding terms. Under the extra assumption that 
there exist a column pair which is the identity, Eq.~\eqref{eq: reversing symm condition} simplifies to Eq.~\eqref{eq: reversing symm condition simple}. This shows that $\mathrm{(i)\implies(ii)}$. 

For the other direction, we show that if Eq.~\eqref{eq: reversing symm condition simple} is satisfied at by the level-1 inflation words, then these sets of equations must also be fulfilled by any power $\varrho^k$ of $\varrho$. %
Remember that, from any arbitrary bijective substitution $\varrho$, we may derive another bijective substitution $\varrho'$ that satisfies the additional condition of having two identity columns in opposing positions by choosing $k=\text{lcm}(|\varrho_0|, |\varrho_{L-1}|)$ and replacing $\varrho$ by its $k$-th power, $\varrho^{\prime}:=\varrho^{k}$. This makes no difference when studying $\mathcal{R}(\mathbb{X}_\varrho)$, because $\varrho$ and $\varrho^{k}$ define the same subshift and the group of reversing symmetries is a property of the hull.

First, we prove an important property of the columns of powers. Fix a power $k\in\mathbb{N}$ and pick a column $\varrho^{ }_i$ of $\varrho^k$, where $0\leqslant i\leqslant L^k-1$. One then has 
$\varrho^{ }_i= \varrho^{ }_{i_0} \cdots \varrho^{ }_{i_{k-1}}$
where $i_{0}i_{1}\cdots i_{k-1}$ is the $L$-adic expansion of $i$ and $\varrho_{i_\ell}$ are columns of the level-$1$ substitution $\varrho$.

The corresponding $L$-adic expansion of $L^{k}-(i+1)$ is then given by \[L^{k}-(i+1)=(L-(i_0+1))\cdots (L-(i_{k-1}+1)).\]
This can easily be shown via the following direct computation
\[
\sum_{j=0}^{k-1}(L-(i_j+1))L^j=\sum_{j=0}^{k-1}(L^{j+1}-L^j)-\sum^{k-1}_{j=0}i_jL^j = L^k-(i+1). 
\]
This implies that if one considers the corresponding column $\varrho^{ }_{L^k-(i+1)}$ one gets that 
\begin{equation}\label{eq:expansion kth level column}
\varrho^{ }_{L^{k}-(i+1)}= \varrho^{ }_{L-(i_0+1)} \cdots \varrho^{ }_{L-(i_{k-1}+1)}. 
\end{equation}
This has two consequences. First, if $\varrho$ has an identity column pair, then all powers of $\varrho$ admit at least one identity column pair. For each power $k$ one just needs to choose $\varrho_j$ with $j=iii\cdots i$, which implies $\varrho_j=\varrho^{k}_i=\text{id}$. By Eq.~\eqref{eq:expansion kth level column}, we also get that $\varrho_{L^{k}-(j+1)}=(\varrho_{L-(i+1)})^k=\text{id}$. In fact, $\varrho^k$ contains at least $2^{k-1}$ pairs of identity columns.

Second, this property allows one to prove that if $\varrho$ satisfies the system of equations in Eq.~\eqref{eq: reversing symm condition simple}, then it is satisfied at all levels, i.e., by all powers of $\varrho$. 
To this end, choose $0\leqslant i\leqslant L^{k}-1$
with $L$-adic expansion $i_{0}i_{1}\cdots i_{k-1}$.
From Eq.~\eqref{eq: reversing symm condition simple} one then obtains
\begin{align*}
\pi^{-1}\circ \varrho^{ }_i\circ \pi &=\pi^{-1}\circ \varrho^{ }_{i_0} \cdots  \varrho^{ }_{i_{k-1}}  \circ \pi =\pi^{-1}\circ \varrho^{ }_{i_0} \pi \pi^{-1}  \cdots \pi \pi^{-1}  \varrho^{ }_{i_{k-1}} \pi  \\
&= \varrho^{ }_{L-(i_0+1)} \cdots \varrho^{ }_{L-(i_{k-1}+1)} = \varrho^{ }_{L^k-(i+1)}.
\end{align*}
Since $i$ is chosen arbitrarily and $\pi$ induces a permutation of the substituted words at all levels, this means it extends to a map $f=\sigma_n\circ m\circ\pi\colon \mathbb{X}_{\varrho}\to \mathbb{X}_{\varrho}$, which by Proposition~\ref{prop: reversor form} is a reversor. This shows $\mathrm{(ii)\implies(i)}$.

To prove the remaining equivalences, note that if $\pi_1,\pi_2\in S_n$ are two permutations satisfying the equality $\pi^{-1}\circ\varrho^{ }_i\circ\pi=\varrho^{ }_{L-(i+1)}$, then we have:
    \[\pi_1\circ\varrho^{ }_{L-(i+1)}\circ\pi_1^{-1} = \varrho^{ }_i\implies (\pi_2\circ\pi_1^{-1})^{-1}\circ\varrho^{ }_i\circ(\pi_2\circ\pi_1^{-1}) = \varrho^{ }_i, \]
that is, $(\pi_2\circ\pi_1^{-1})\in\operatorname{cent}_{S_n}(\varrho_i)$. As a consequence, $\pi_1$ belongs to the right coset $\operatorname{cent}_{S_n}(\varrho_i)\pi_2$ for any choice of $\pi_1,\pi_2$, and, since right cosets are either equal or disjoint, this means that all solutions of Eq.~\eqref{eq: reversing symm condition simple}, for a fixed $i$,
lie in the same right coset of $\operatorname{cent}_{S_n}(\varrho_i)$. Reciprocally, if $\pi$ satisfies Eq.~\eqref{eq: reversing symm condition simple} and $\gamma\in\operatorname{cent}_{S_n}(\varrho_i)$, it is easy to verify that $\gamma\circ\pi$ satisfies Eq.~\eqref{eq: reversing symm condition simple} as well. Thus, the set of solutions of this equation is either empty or the aforementioned uniquely defined right coset.

Thus, suppose that $\pi$ satisfies Eq.~\eqref{eq: reversing symm condition simple} for all $0\leqslant i\leqslant L-1$. The set of solutions for each $i$ equals the unique coset $\operatorname{cent}_{S_n}(\varrho_i)\pi$, and thus the set of all permutations that satisfy Eq.~\eqref{eq: reversing symm condition simple} for all $i$ is exactly the intersection of all these cosets, i.e. $\bigcap_{i=0}^{L-1}\operatorname{cent}_{S_n}(\varrho_i)\pi$.
Taking $\kappa^{ }_i=\pi$ for all $i$, we see that this is exactly the set $K$ from \eqref{eq: coset intersection}. Evidently, $\pi$ belongs to this intersection, and so we conclude that $\mathrm{(ii)\implies(iii)}$.

As stated before, our choice of $\kappa^{ }_i$ ensures that the set $\operatorname{cent}_{S_n}(\varrho^{ }_i)\kappa^{ }_i$ is exactly the set of solutions of Eq.~\eqref{eq: reversing symm condition simple} for a given $i$; thus, any permutation $\pi$ that satisfies all of these equalities must be in all of these cosets and thus in the intersection \eqref{eq: coset intersection}, which is therefore non-empty. This shows that $\mathrm{(iii)\implies(ii)}$, concluding the proof.
\end{proof}

The following general criterion on when a letter-exchange map generates  a reversor is given in \cite{BRY}. 

\begin{lemma}[{\cite[Lem.~2]{BRY}}]\label{lem: BRY lemma}
Let $\varrho$ be a primitive constant-length substitution of height $1$ and column number $c_{\varrho}$. Suppose that $\varrho$ is strongly injective. Then, a permutation $\pi\colon \mathcal{A}\to\mathcal{A}$ generates a reversor $f\in\mathcal{R}(\mathbb{X}_{\varrho})$ if and only if
\begin{enumerate}
\item $ab\in\mathcal{L}_{\varrho}^2\implies \pi(ba)\in \mathcal{L}_{\varrho}^2$
\item $(\pi\circ\varrho^{c_{\varrho}!})(ab)=(\varrho^{c_{\varrho}!}\circ \pi )(ba)$
\end{enumerate}
for each $ab\in\mathcal{L}^2_{\varrho}$. \qed
\end{lemma}

For a primitive, aperiodic and bijective $\varrho$, one has 
$c_{\varrho}=|\mathcal{A}|$. Moreover, $\varrho$ is always strongly injective. Note that Theorem~\ref{thm: reversing symmetries iff} implies conditions \textbf{(1)} and \textbf{(2)} in Lemma~\ref{lem: BRY lemma}. The first one immediately follows from primitivity, and the fact that any legal word $ab$ appears in some level-$n$ superword, which is sent to another level-$n$ superword by $\pi\circ m$, which guarantees the legality of $\pi(ba)$. The second follows from the fact that $\pi$ is compatible with superwords of all levels. In fact, one has $\big(\pi\circ \varrho^{n}\big)(ab)=\big(\varrho^{n}\circ\pi\big)(ba)$ for all $n\in\mathbb{N}$.

\begin{remark}
    It is a known fact from group theory that, if $g_1,\dotsc,g_r$ are elements of a group $G$ and $H_1,\dots,H_r$ are subgroups of this group, the intersection of cosets $\bigcap_{i=1}^r g_iH_i$ is either empty or a coset of $\bigcap_{i=1}^r H_i$. In this case, the latter intersection is exactly the group of non-trivial standard symmetries modulo a shift (letter exchanges), and thus, if there exist non-trivial reversing symmetries, these must all belong to a single coset of the group of valid letter exchanges. This is consistent with the fact that $\mathcal{R}(\mathbb{X}_\varrho)$ is at most an index $2$ group extension of $\mathcal{S}(\mathbb{X}_\varrho)$. \exend 
\end{remark}
Item (3) in Theorem \ref{thm: reversing symmetries iff}
provides an explicit algorithm to compute the group of permutations $\pi$ which define extended symmetries, which is a counterpart to that in Section~\ref{sec: symmetries} for standard symmetries. As stated previously, the centralisers $\operatorname{cent}_{S_n}(\varrho^{ }_j)$ can be computed for each column using Fact~\ref{fact:permutation_groups_rep}, and thus the problem reduces to obtaining a suitable candidate for each $\kappa^{ }_i$, which once again can be done by an application of Fact~\ref{fact:permutation_groups_rep}. The algorithm is as follows: 

\vspace{1em}
\hrule
\vspace{0.3em}
{\noindent\small\textbf{Algorithm.} Assuming that $\varrho$ is a primitive, bijective, aperiodic substitution, the following algorithm computes the set $K$ of permutations that induce reversors, which determines $\mathcal{R}(\mathbb{X}^{ }_\varrho)$.
\vspace{1em}
\begin{itemize}
    \item \textbf{Input:} $\varrho$ is a length-$L$ bijective substitution, represented either as a function or a set of columns.
    \item \textbf{Output:} A (finite) set of permutations $K$, either empty or a coset of the group $C$ computed by the previous algorithm, so that $\mathcal{R}(\mathbb{X}^{ }_\varrho)/\langle\sigma\rangle\simeq C\cup K$ (i.e. $\mathcal{R}(\mathbb{X}^{ }_\varrho)\simeq\mathbb{Z}\rtimes_\varphi (C\cup K)$, with $\varphi(g,n) = n$ if $g\in C$, and $-n$ if $g\in K$).
\end{itemize}
\vspace{1em}
\begin{itemize}
    \item[(1)] Let $N$ be the least positive integer which ensures that two opposite columns of $\varrho^N$ are the identity map. This can be computed as: \[N=\min\left\{\operatorname{lcm}(\operatorname{ord}(\varrho^{ }_i),\operatorname{ord}(\varrho^{ }_{L-(i+1)}))\::\:0\leqslant i \leqslant N/2 \right\}.\]
    \item[(2)] For each $0\leqslant i\leqslant N/2$, compute $\kappa^{ }_i$ via the following subroutine:
        \begin{itemize}
            \item[(2.i)] If $\varrho^{ }_i$ and $\varrho^{ }_{L-(i+1)}$ are non-conjugate (i.e., their cycle decomposition has a different number of cycles of some length), stop the algorithm, as reversors do not exist (see Theorem~\ref{thm: reversing symmetries iff}).
            \item[(2.ii)] Sort the cycles from the disjoint cycle decomposition of $\varrho^{ }_i$ by increasing order of length. Using this as a basis, by appropriately sorting the elements of each cycle in this decomposition, define a total order relation $<$ on $\mathcal{A}$, given by, say, $a_1<\dots<a_n$, such that all of the elements of a given cycle come before the elements of the following cycle, in the sorting by left. Do the same for $\varrho^{ }_{L-(i+1)}$, defining a corresponding total order $<'$ given by $b_1<'\dots<'b_n$. This ensures that there are cycle decompositions of both permutations such that the corresponding cycles, ordered from left to right, have the same length, as follows:
            \begin{align*}
                \varrho^{ }_i &= (a_1\,\dotsc\,a_j)(a_{j+1}\,\dotsc\,a_{j'})\cdots(a_{j''+1}\,\dotsc\,a_n), \\
                \varrho^{ }_{L-(j+1)} &= (b_1\,\dotsc\,b_j)(b_{j+1}\,\dotsc\,b_{j'})\cdots(b_{j''+1}\,\dotsc\,b_n),
            \end{align*}
            with $1\leqslant j\leqslant j'\leqslant\dotsc\leqslant j''\leqslant n$.
           
            \item[(2.iii)] Define: \[\kappa^{ }_i=\begin{pmatrix}
                a_1 & a_2 & \cdots & a_n \\
                b_1 & b_2 & \cdots & b_n
            \end{pmatrix},\qquad\kappa^{ }_{L-(i+1)} = \kappa_i^{-1}. \]
            
        \end{itemize}
    \item[(3)] Compute each centraliser $C^{(i)} = \operatorname{cent}_{S_n}(\varrho^{ }_i)$, using the same procedure as in the computation of $\mathcal{S}(\mathbb{X}_\varrho)$. \item[(4)] Return $K = \bigcap_{i=1}^N C^{(i)}\kappa_i$. Any element of $K$ induces a reversor; if $K$ is empty, reversors do not exist.
\end{itemize}}
\hrule
\vspace{1em}
Any programming environment with suitable data structures (e.g. computer algebra systems such as \texttt{Sagemath}${}^\text{\textregistered}$ or \texttt{Mathematica}${}^\text{\textregistered}$) is amenable to the implementation of this algorithm, providing effective procedures to entirely characterise the groups $\mathcal{S}(\mathbb{X}_\varrho)$ and $\mathcal{R}(\mathbb{X}_\varrho)$ from a suitable description of the substitution $\varrho$, e.g. using a dictionary.

\subsection{Higher-dimensional subshifts}

Now, we turn our attention to the situation in higher dimensions. The \emph{extended symmetry group} of a $\mathbb{Z}^{d}$-shift is defined as $
\mathcal{R}(\mathbb{X})=\text{norm}_{\text{Aut}(\mathbb{X})} (\mathcal{G})$, 
where now $\mathcal{G}=\left\langle \sigma_{e_1},\ldots,\sigma_{e_d}\right\rangle\simeq \mathbb{Z}^{d}$;  see \cite{BRY,B,BBHLN}. In this more general context, an \emph{extended symmetry} is a element $f\in\mathcal{R}(\mathbb{X})\setminus \mathcal{S}(\mathbb{X})$. 

Similar to standard symmetries, there is a direct generalisation of the characterisation of extended symmetries from Proposition~\ref{prop:reversing necessary} and the subsequent theorem to the higher-dimensional setting, which is given by the following. 

\begin{proposition}
Let $\varrho$ be an aperiodic, primitive, bijective, block substitution in $\mathbb{Z}^{d}$. Then any extended symmetry $f\in\mathcal{R}(\mathbb{X}_{\varrho})\setminus \mathcal{S}(\mathbb{X}_{\varrho})$ 
must be (up to a shift) a composition of a letter exchange map and a rearrangement function $f_A$ given by $f_A(x)_{\boldsymbol{n}} = x_{A\boldsymbol{n}}$, where $A\in\mathrm{GL}(d,\mathbb{Z})$, with $A\neq \mathbb{I}$. \qed
\end{proposition}

For shifts generated by bijective rectangular substitutions one has the following restriction on the linear component $A$ of an extended symmetry $f$. 

\begin{theorem}[{\cite[Thm.~18]{B}}]\label{thm: hyperoctahedral HD}
Let $\varrho$ an aperiodic, primitive, bijective rectangular substitution in $\mathbb{Z}^{d}$. One then has
\[
\mathcal{R}(\mathbb{X}_{\varrho})/\mathcal{S}(\mathbb{X}_{\varrho})\simeq 
P\leqslant W_d, 
\]
where $W_d \simeq C^d_2\rtimes S_d$ is the $d$-dimensional hyperoctahedral group, which represents the symmetries of the $d$-dimensional cube. \qed
\end{theorem}

With this, one can show that all extended symmetries of such subshifts are of finite order. The proof of the following result is patterned from \cite[Prop.~2]{BR},  which deals with the order of reversors of an automorphism $h$ of a general  dynamical system with $\text{ord}(h)=\infty$; compare \cite{Goodson}.

\begin{proposition}\label{prop:Rorder}
Let $\mathbb{X}_{\varrho}$ be the same as above with symmetry group $\mathcal{S}(\mathbb{X}_{\varrho})=\mathbb{Z}^{d}\times G$. Let $f\in \mathcal{R}(\mathbb{X}_{\varrho})\setminus \mathcal{S}(\mathbb{X}_{\varrho})$ be an extended symmetry, whose associated matrix is $A\in W_d$. Then $\textnormal{ord}(f)$ divides $\textnormal{ord}(A)\cdot|G|$. Moreover, $\textnormal{ord}(f)\leqslant 2|G|\cdot\max\left\{\textnormal{ord}(\tau)\mid \tau\in S_d\right\}$.
\end{proposition}

\begin{proof}
Under the given assumptions, $
f\circ \sigma_{\boldsymbol{m}}  \circ f^{-1} =\sigma_{A\boldsymbol{m}}$ holds
for all $\boldsymbol{m}\in\mathbb{Z}^{d}$, which yields
\begin{align}
f^{\ell}\circ\sigma_{\boldsymbol{m}}\circ f^{-\ell}&=\sigma_{A^{\ell}\boldsymbol{m}}\label{eq:ell and r}\\
f\circ\sigma_{n\boldsymbol{m}}\circ f^{-1}&=\sigma_{nA\boldsymbol{m}} \label{eq:n and r}
\end{align}
for all $\ell,n\in \mathbb{N}$. Choosing $\ell=\text{ord}(A)$, Eq.~\eqref{eq:ell and r} gives $f^{\text{ord}(A)}\in \mathcal{S}(\mathbb{X}_{\varrho})$. From Theorem~\ref{thm:symm_group}, $f^{\text{ord}(A)}=\sigma_{\boldsymbol{p}}\circ \pi$, for some $\boldsymbol{p}\in\mathbb{Z}^d$ and  letter-exchange map $\pi$. From the direct product structure of the symmetry group, one has $\sigma_{\boldsymbol{p}}\circ \pi=\pi\circ\sigma_{\boldsymbol{p}}$, which implies $f^{\text{ord}(A)\cdot|G|}=\sigma_{|G|\boldsymbol{p}}\circ \pi^{|G|}=\sigma_{|G|\boldsymbol{p}}$. Using the two equations above, one gets $f^{\text{ord}(A)\cdot|G|}=\sigma_{|G|A^{\ell}(\boldsymbol{p})}$ for all $\ell\in\mathbb{N}$. Since $f$ is an extended symmetry, $A\neq \mathbb{I}$. 
Next we show that $\boldsymbol{p}$ cannot be an eigenvector of $A$. 

Suppose $A\boldsymbol{p}=\boldsymbol{p}$ with $\boldsymbol{p}\neq\boldsymbol{0}$. Note that $f^{-\text{ord}(A)|G|}=\sigma_{-|G|\boldsymbol{p}}$. From Eqs.~\eqref{eq:ell and r} and \eqref{eq:n and r}, one also has 
$f^{-1}\circ \sigma_{|G|A^{-1}\boldsymbol{p}}\circ f=\sigma_{-|G|\boldsymbol{p}}$, which implies $A^{-1}\boldsymbol{p}=-\boldsymbol{p}$, contradicting the assumption on $\boldsymbol{p}$. 
 Since $\text{ord}(\sigma_{\boldsymbol{p}})=\infty$, this forces $\boldsymbol{p}=\boldsymbol{0}$ and hence $f^{\text{ord}(A)\cdot|G|}=\text{id}$ from which the first claim is immediate.  The upper bound for the order follows from the upper bound for the order of the elements of the hyperoctahedral group $W_d$; see \cite{Baake}. 
\end{proof}

Due to the fact that $\mathcal{R}(\mathbb{X}_\varrho)$ is (possibly) a larger extension of $\mathcal{S}(\mathbb{X}_\varrho)$ (that is, the corresponding quotient can have up to $2^dd!-1$ non-trivial elements instead of just one), we would end up with a much larger number of equations of the form of Eq.~\eqref{eq: reversing symm condition}, one for each element of the hyperoctahedral group $W_d$ except the identity. This leads us to another problem of  different nature: if the rectangle $R$, which is the support of the level-$1$ supertiles of $\varrho$, is not a cube in $\mathbb{Z}^d$, some symmetries from $W_d$ may not be compatible with $R$, i.e., they may map $R$ to a different rectangle that is not a translation of $R$, so the corresponding equation does not have a proper meaning (as it may compare an existing column with a non-existent one).

\begin{figure}[!ht]
	\centering
	\begin{tikzpicture}[scale=0.5]
		\begin{scope}
			\draw[fill=white] (-0.5,-0.5) rectangle (0.5,0.5);
			\node at (1.5,0) {$\mapsto$};
			\begin{scope}[xshift=2.5cm,yshift=-1cm]
				\draw[fill=white] (0,0) rectangle (1,1) rectangle (2,2) rectangle (3,1) rectangle (4,0);
				\draw[fill=black] (0,2) rectangle (1,1) rectangle (2,0) rectangle (3,1) rectangle (4,2);
			\end{scope}
		\end{scope}
		\begin{scope}[xshift=10cm]
			\draw[fill=black] (-0.5,-0.5) rectangle (0.5,0.5);
			\node at (1.5,0) {$\mapsto$};
			\begin{scope}[xshift=2.5cm,yshift=-1cm]
				\draw[fill=black] (0,0) rectangle (1,1) rectangle (2,2) rectangle (3,1) rectangle (4,0);
				\draw[fill=white] (0,2) rectangle (1,1) rectangle (2,0) rectangle (3,1) rectangle (4,2);
			\end{scope}
		\end{scope}
	\end{tikzpicture}
	\caption{A non-square substitution that generates the two-dimensional Thue-Morse hull.}
	\label{fig:alternate_thuemorse}
\end{figure}
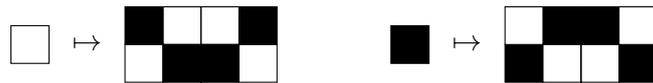

This could be taken as a suggestion that such symmetries cannot actually happen, imposing further limitations on the quotient $\mathcal{R}(\mathbb{X}_\varrho)/\mathcal{S}(\mathbb{X}_\varrho)$. Interestingly, this is not actually the case. For instance, consider the  two-dimensional rectangular substitution from Figure \ref{fig:alternate_thuemorse}. As the support for this substitution is a $4\times 2$ rectangle, we could guess that this substitution is incompatible with rotational symmetries or reflections along a diagonal axis, which would produce a $2\times 4$ rectangle instead. However, further examination shows that the hull generated by this substitution is actually the same as the hull of the two-dimensional Thue--Morse substitution as seen in e.g. \cite{B}, which is compatible with every symmetry from $W_2=D_4$. Thus, only geometrical considerations are not enough to exclude candidates for extended symmetries.

Fortunately, there is a subcase of particular interest in which this geometrical intuition is actually correct, which involves an arithmetic restriction on the side lengths of the support rectangle $R$. It turns out that coprimality of the side lengths is a sufficient condition (although it can be weakened even further) to rule out such symmetries, e.g.~there are no extended symmetries compatible with rotations when $R$ is a, say, $2\times 5$ rectangle. To be precise:

\begin{theorem}\label{thm: exclude HD symmetries}
   Let $\varrho\colon\mathcal{A}\to\mathcal{A}^R$ be a bijective rectangular substitution with faithful associated shift action. Suppose that $R=[\boldsymbol{0},\boldsymbol{L}-\boldsymbol{1}]$ with $\boldsymbol{L}=(L_1,\dotsc,L_d)$ (that is, $R$ is a $d$-dimensional rectangle with side lengths $L_1, L_2,\dotsc, L_d$) and that for some indices $i,j$ there is a prime $p$ such that $p\mid L_j$ but $p\nmid L_i$, i.e. $L_i$ and $L_j$ have different sets of prime factors. Let $A\in W_d\leqslant\mathrm{GL}(d,\mathbb{Z})$ and suppose that $A$ is the underlying matrix associated to an extended symmetry $f\in\mathcal{R}(\mathbb{X}_{\varrho})$. Then $A_{ij}=A_{ji}=0$.
\end{theorem}

The underlying idea is that, if $A\in W_d$ induces a valid extended symmetry for some substitution $\varrho$ with support $U$, we can find another substitution $\eta$ with support $A\cdot U$ (up to an appropriate translation) such that $\mathbb{X}_\varrho = \mathbb{X}_\eta$, and then we use the known factor map from an aperiodic substitutive subshift onto an associated odometer to rule out certain matrices $A$. Similar exclusion results have been studied by Cortez and Durand \cite{CD}.

\begin{proof}
    Let $\varphi\colon\mathbb{X}_{\varrho}\twoheadrightarrow\mathbb{Z}_{L_1}\times\cdots\times\mathbb{Z}_{L_d}=\mathbb{Z}_{\boldsymbol{L}}$ be the standard factor map from the substitutive subshift to the corresponding product of odometers. It is known \cite[Thm.~5]{BRY} that, for any extended symmetry $f\colon\mathbb{X}_\varrho\to\mathbb{X}_\varrho$ with associated matrix $A$, there exists $\boldsymbol{k}_f = (k_1,\dotsc,k_d)\in\mathbb{Z}_{\boldsymbol{L}}$ and a group automorphism $\alpha_f\colon \mathbb{Z}_{\boldsymbol{L}}\to \mathbb{Z}_{\boldsymbol{L}}$ satisfying the following equation:
		\begin{equation}\label{eq:factor_symmetry_representation}
		    \varphi(f(x)) = \boldsymbol{k}_f + \alpha_f(\varphi(x)),
		\end{equation}
	where $\alpha_f$ is the unique extension of the map $\boldsymbol{n}\mapsto A\boldsymbol{n}$, defined in the dense subset $\mathbb{Z}^d$, to $\mathbb{Z}_{\boldsymbol{L}}$. In particular, for any $\boldsymbol{n}\in\mathbb{Z}^d$, if $f=\sigma_{\boldsymbol{n}}$ is a shift map, then $\boldsymbol{k}_{\sigma_{\boldsymbol{n}}}=\boldsymbol{n}$ and $\alpha_{\sigma_{\boldsymbol{n}}} = \mathrm{id}_{\mathbb{Z}_{\boldsymbol{L}}}$.
		
	Now, consider the sequence $\boldsymbol{h}_m = L_i^m\boldsymbol{e}_i$, and suppose $A_{ji} = \pm1$. Equivalently, $A\boldsymbol{e}_i = \pm \boldsymbol{e}_j$, since $A$ is a signed permutation matrix. Without loss of generality, we may assume the sign to be $+$. One has $L_i^m\xrightarrow{m\to\infty}0$ in the $L_i$-adic topology, and thus $\varphi(\sigma_{\boldsymbol{h}_m}(x)) = \boldsymbol{h}_m+\varphi(x) \xrightarrow{m\to\infty}\varphi(x)$, as it does so componentwise. By compactness, we may take a subsequence $\boldsymbol{h}_{\beta(m)}$ such that $\sigma_{\boldsymbol{h}_{\beta(m)}}(x)$ converges to some $x^*$; then, as the factor map $\varphi$ is continuous, we have $\varphi(x^*) = \varphi(x)$.

    Eq.~(\ref{eq:factor_symmetry_representation}) and this last equality imply that $\varphi(f(x))=\varphi(f(x^*))$ as well. Writing $x^*$ as a limit, we obtain from continuity that
        \begin{align*}
            \varphi(x^*) & = \lim_{m\to\infty}\varphi(f(\sigma_{\boldsymbol{h}_{\beta(m)}}(x)))= \lim_{m\to\infty}\varphi(\sigma_{A\boldsymbol{h}_{\beta(m)}}(f(x)))\\
            &= \varphi(x) + \lim_{m\to\infty}A\boldsymbol{h}_{\beta(m)}= \varphi(x) + \lim_{m\to\infty}L_i^{\beta(m)}A\boldsymbol{e}_{i} \\
            \implies \lim_{m\to\infty}L_i^{\beta(m)}\boldsymbol{e}_j &= \varphi(x^*) - \varphi(x)= \boldsymbol{0}.
        \end{align*}
    The last equality implies that, in the topology of $\mathbb{Z}_{L_j}$, the sequence $L_i^{\beta(m)}$ converges to $0$. However, since there is a prime $p$ that divides $L_j$ but not $L_i$, due to transitivity we must have $L_j\nmid L_i^n$ for all $n$, as otherwise $p\mid L_i^n$ and thus $p\mid L_i$. Thus, in base $L_j$, the last digit of $L_i^{\beta(m)}$ is never zero, and thus $L_i^{\beta(m)}$ remains at fixed distance $1$ from $\boldsymbol{0}$ (in the $L_j$-adic metric), contradicting this convergence. Thus, $A_{ji}$ cannot be $1$ and must necessarily equal $0$. For $A_{ij}$, the same reasoning applies to $f^{-1}$. Since $A$ is a signed permutation matrix, $A_{ij} = \pm 1$ would imply $(A^{-1})_{ji} = \pm 1$, again a contradiction.
\end{proof}

We now proceed to the generalisation of Theorem~\ref{thm: reversing symmetries iff} in higher dimensions. 
As before, for a block substitution $\varrho$, we have $R=\prod_{i=1}^{d}[0,L_{i}-1]$, with $L_i\geqslant 2$ and the expansive map $Q=\textnormal{diag}(L_1,L_2,\ldots,L_d)$. 
Let $A\in W_d\leqslant \text{GL}(d,\mathbb{Z})$ be a signed permutation matrix.  First, we assume that the location of a tile in any supertile is given by the location of its centre. 
Define the affine map $A^{(1)}\colon R\to R$ via $
A^{(1)}(\boldsymbol{i})=A(\boldsymbol{i}-\boldsymbol{x}_1)+|A|\boldsymbol{x}_1$
where $\boldsymbol{i}\in R$ and $\boldsymbol{x}_1=Q\boldsymbol{v}-\boldsymbol{v}$ with $\boldsymbol{v}=\frac{1}{2}(1,1,\ldots,1)^T$. Here, $(|A|)_{ij}=|A_{ij}|$.  The vector $|A|\boldsymbol{x}_1$ is the translation needed to shift the centre of the supertile to the origin, which we will need before applying the map $A$ and shifting it back again. 
We extend $A^{(1)}$ to any level-$k$ supertile by defining the map 
$A^{(k)}\colon R^{(k)}\to R^{(k)}$ given by 
\begin{equation}\label{eq: level-k g map} 
A^{(k)}(\boldsymbol{i})=A(\boldsymbol{i}-\boldsymbol{x}_k)+|A|\boldsymbol{x}_k,
\end{equation}
with $\boldsymbol{i}\in R^{(k)}$ and $\boldsymbol{x}_k=Q^{k}\boldsymbol{v}-\boldsymbol{v}$. Here  $ R^{(k)}:=\prod_{i=1}^{d}[0,L_i^k-1]$ is the set of locations of tiles in a level-$k$ supertile.

\begin{example}
Let $\varrho$ be a two-dimensional block substitution with $Q=\begin{psmallmatrix}
2 & 0\\
0 & 2
\end{psmallmatrix}$ and $A$ be the counterclockwise rotation by $90$ degrees, with corresponding matrix $A=\begin{psmallmatrix}
0 &-1\\
1& 0
\end{psmallmatrix}$. Consider the level-$3$ supertile and let $\boldsymbol{i}=(7,3)^{T}\in R^{(3)}$, with $Q$-adic expansion $\boldsymbol{i}\,\widehat{=}\,\boldsymbol{i}_2\boldsymbol{i}_1\boldsymbol{i}_0$. Here one has $\boldsymbol{i}_0=\boldsymbol{i}_1=\boldsymbol{e}_1+\boldsymbol{e}_2$ and $\boldsymbol{i}_2=\boldsymbol{e}_1$. One then gets $A^{(3)}(\boldsymbol{i})=(4,7)^{T}$; see Figure~\ref{fig: HD location set}. One can check that $\sum_{j=0}^{2} Q^{j}(A^{(1)}(\boldsymbol{i}_j))=A^{(3)}(\boldsymbol{i})$. \exend
\end{example}

\begin{figure}[h]
\begin{center}
\includegraphics[scale=1.1]{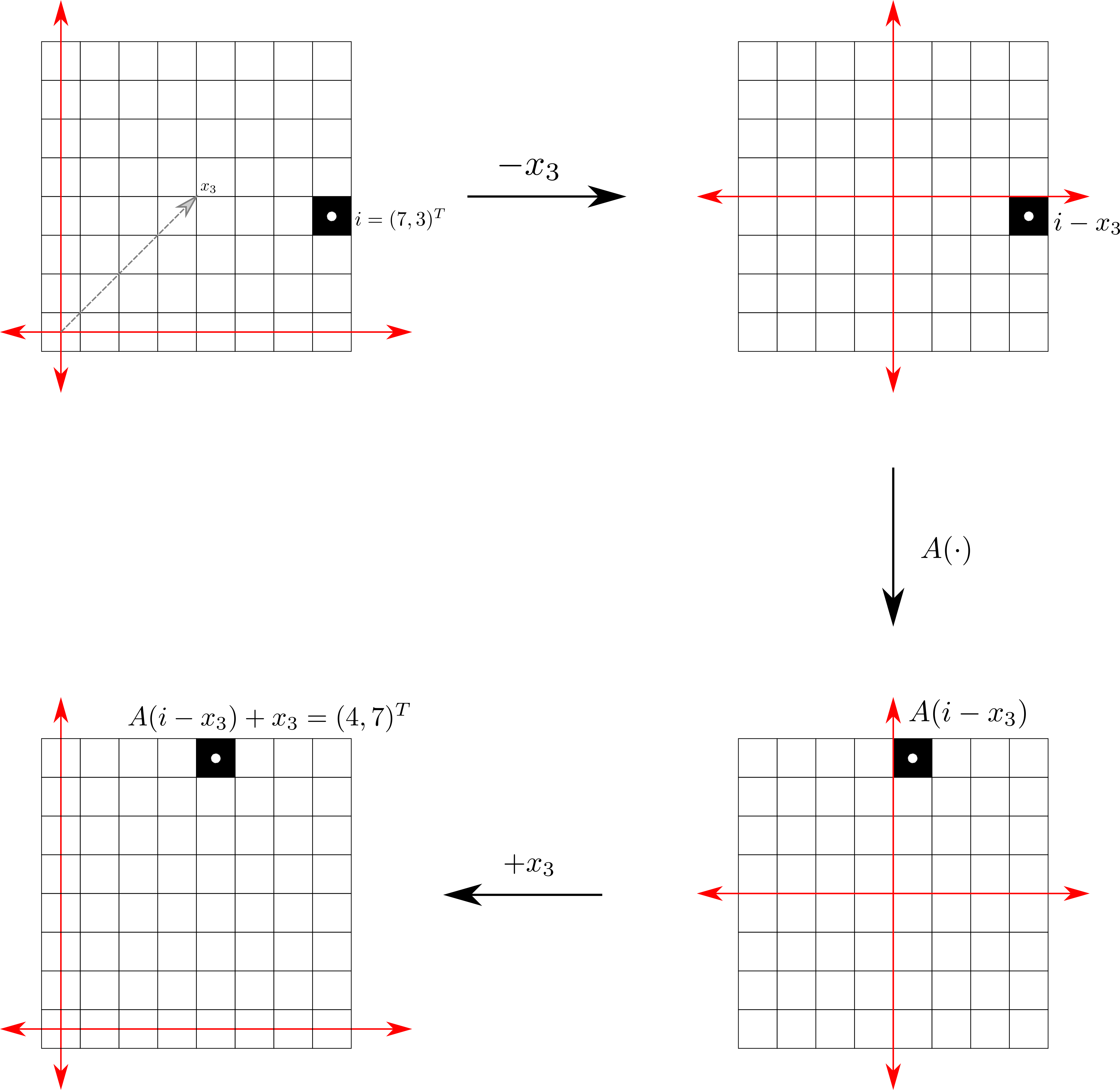}
\end{center}
\caption{The transformation of a marked level-3 location set $R^{(3)}$ under the map $A^{(3)}$.}\label{fig: HD location set}
\end{figure}

The following result is the analogue of Theorem~\ref{thm: reversing symmetries iff} in $\mathbb{Z}^{d}$.

\begin{theorem}\label{thm: extended symmetries iff}
Let $\varrho$ be an aperiodic, primitive, bijective block substitution $\varrho\colon \mathcal{A} \rightarrow \mathcal{A}^{R}$.
Let $W_d$ be the $d$-dimensional hyperoctahedral group and let $A\in W_d$. 
Suppose there exists $\boldsymbol{\ell}\in R$ such that $\varrho^{ }_{\boldsymbol{\ell}^{\prime}}=\textnormal{id}$ for all $\boldsymbol{\ell}^{\prime}\in \textnormal{Orb}_A(\boldsymbol{\ell})$.  
Assume further that $[A,Q]=0$ and $|A|\boldsymbol{x}_1=\boldsymbol{x}_1$.
Then $\pi$, together with $A$, gives rise to an extended symmetry  $f\in \mathcal{R}(\mathbb{X}_\varrho)$ if and only if
\begin{equation}\label{eq: HD sufficient}
\pi^{-1} \circ  \varrho^{ }_{\boldsymbol{i}} \circ \pi= \varrho^{ }_{A^{(1)}(\boldsymbol{i})}
\end{equation}
for all $\boldsymbol{i}\in R$. 
\end{theorem}

\begin{proof}

Most parts of the proof mimics those of the proof of Theorem~\ref{thm: reversing symmetries iff}, where one replaces the mirroring operation $m$
with a more general map $A\in W_d$. One then gets an analogous system of equations, as in those coming from Eq.~\eqref{eq: compare tau}. Using 
this, one can show the necessity direction. 

To prove sufficiency, we show that
if Eq.~\eqref{eq: HD sufficient} is satisfied for all $\boldsymbol{i}\in R$, then it also holds for all positions in any level-$k$ supertile. 
Let $\boldsymbol{i}\in R^{(k)}$, which admits the unique $Q$-adic expansion given by $\boldsymbol{i}\,\widehat{=}\,\boldsymbol{i}_{k-1}\boldsymbol{i}_{k-2}\cdots \boldsymbol{i}_1 \boldsymbol{i}_0$, i.e., $\boldsymbol{i}=\sum_{j=0}^{k-1}Q^{j}(\boldsymbol{i}_j)$. We now show that the $Q$-adic expansion of $A^{(k)}(\boldsymbol{i})$ is given by 
$A^{(k)}(\boldsymbol{i})\,\widehat{=}\, A^{(1)}(\boldsymbol{i}_{k-1})A^{(1)}(\boldsymbol{i}_{k-2})\cdots A^{(1)}(\boldsymbol{i}_{0})$. Plugging in the expansion of $\boldsymbol{i}$ into Eq.~\eqref{eq: level-k g map}, one gets $
A^{(k)}(\boldsymbol{i})=\left(\sum_{j=0}^{k-1} AQ^{j}(\boldsymbol{i}_j)\right)-A\boldsymbol{x}_k+\boldsymbol{x}_k$. 
On the other hand, one also has
\begin{align*}
\sum_{j=0}^{k-1}Q^{j}(A^{(1)}(\boldsymbol{i}_j))&=\sum_{j=0}^{k-1} Q^{j}\big(A(\boldsymbol{i}-Q\boldsymbol{v}+\boldsymbol{v})+Q\boldsymbol{v}-\boldsymbol{v}\big)\\
&= \sum_{j=0}^{k-1} Q^jA(\boldsymbol{i}_j)+\sum_{j=0}^{k-1}\left( -AQ^{j+1}\boldsymbol{v}+AQ^j\boldsymbol{v} \right)+ \sum_{j=0}^{k-1} \left(Q^{j}\boldsymbol{v}-Q^{j}\boldsymbol{v}\right)\\
&=\sum_{j=0}^{k-1} AQ^j(\boldsymbol{i}_j)\underbrace{-AQ^{k}\boldsymbol{v}+A\boldsymbol{v}}_{-A\boldsymbol{x}_k} + \underbrace{Q^{k}\boldsymbol{v}-\boldsymbol{v}}_{\boldsymbol{x}_k}\\
&=A^{(k)}(\boldsymbol{i}),
\end{align*}
where the penultimate equality follows from $[A,Q]=0$ and the evaluation of the two telescoping sums. As in Theorem~\ref{thm: reversing symmetries iff}, one then obtains 
\[
\pi^{-1}\circ \varrho^{ }_{\boldsymbol{i}}\circ
\pi=\pi^{-1}\circ \varrho^{ }_{\boldsymbol{i}_{k-1}}\circ\varrho^{ }_{\boldsymbol{i}_{k-2}}\circ\cdots \varrho^{ }_{\boldsymbol{i}_{0}}\circ\pi=\varrho^{ }_{A^{(k)}(\boldsymbol{i})},
\] 
whenever $\boldsymbol{i}\,\widehat{=}\,\boldsymbol{i}_{k-1}\boldsymbol{i}_{k-2}\cdots \boldsymbol{i}_0$ and $\pi^{-1}\circ \varrho^{ }_{\boldsymbol{i}_s}\circ \pi =\varrho^{ }_{A^{(1)}(\boldsymbol{i}_s)}$ for all $\boldsymbol{i}_s\in R$, which finishes the proof.
\end{proof}

\begin{remark}
The conditions $[A,Q]=0$ and $|A|\boldsymbol{x}_1=\boldsymbol{x}_1$ in Theorem~\ref{thm: extended symmetries iff} are automatically satisfied if $\varrho$ is a cubic substitution, i.e., $L_i=L$ for all $1\leqslant i\leqslant d$, which means 
one can use Eq.~\eqref{eq: HD sufficient} to check whether a given letter-exchange map works for any $A\in W_d$. 
For general $\varrho$, these relations are only satisfied for certain $A\in W_d$, e.g. reflections along coordinate axes, which means one needs a different tool to ascertain whether it is possible for other rigid motions to generate extended symmetries. For example, one can use Theorem~\ref{thm: exclude HD symmetries} to exclude some symmetries. \exend
\end{remark}

Before we proceed, we need a higher-dimensional generalisation of Proposition~\ref{prop: aperiodicity proximal pair} regarding aperiodicity. For this, we use the following result, which is formulated in terms of Delone sets. Here, $\mathbb{S}^{d-1}$ is the unit sphere in $\mathbb{R}^d$.

\begin{theorem}[{\cite[Thm.~5.1]{BG}}]\label{thm: aperiodicity Delone}
Let $\mathbb{X}(\varLambda)$ be the continuous hull of a repetitive Delone set $\varLambda\subset\mathbb{R}^d$. Let $\left\{\boldsymbol{b}_i\in \mathbb{S}^{d-1}\mid 1\leqslant i\leqslant d \right\}$ be a basis of $\mathbb{R}^d$ such that for each $i$, there are two distinct elements of $\mathbb{X}(\varLambda)$ which agree on the half-space $\left\{\boldsymbol{x}\mid\left\langle \boldsymbol{b}_i|\boldsymbol{x} \right\rangle >\alpha_i\right\}$ for some $\alpha_i\in\mathbb{R}^d$. Then one has that $\mathbb{X}(\varLambda)$ is aperiodic.  \qed
\end{theorem}

The proof of the previous theorem relies on the generalisation of the notion of proximality for tilings and Delone sets in $\mathbb{R}^{d}$, which is proximality along $\boldsymbol{s}\in\mathbb{S}^{d-1}$; see \cite[Sec.~5.5]{BG} for further details. 
Note that from a $\mathbb{Z}^d$-tiling generated by a rectangular substitution, one can derive a (coloured) Delone set $\varLambda$ by choosing a consistent control point for each cube (usually one of the corners or the centre). Primitivity guarantees that $\varLambda$ is repetitive and the notion of proximality extends trivially to coloured Delone sets using the same metric. The two hulls $\mathbb{X}(\varLambda)$ and $\mathbb{X}_{\varrho}$ are then mutually locally derivable, and the aperiodicity of one implies that of the other. We then have a sufficient criterion for the aperiodicity of $\mathbb{X}_{\varrho}$ in higher dimensions.

\begin{proposition}\label{prop: aperiodicity HD proximal}
Let $\varrho\colon \mathcal{A}\to \mathcal{A}^{R}$ be a $d$-dimensional rectangular substitution which is bijective and primitive. If there exist two legal blocks $u,v\in \mathcal{L}$ of side-length $2$ in each direction such that $u$ and $v$ disagrees at exactly one position and coincides at all other positions, then the hull  $\mathbb{X}_{\varrho}$ is aperiodic. 
\end{proposition}

\begin{proof}
The proof proceeds in analogy to Proposition~\ref{prop: aperiodicity proximal pair}. Here we choose the appropriate power to be 
 $k=\text{lcm}\left\{|\varrho^{ }_{\boldsymbol{r}}| \colon \boldsymbol{r}=\sum^{d}_{i=1} r_i\boldsymbol{e}_i, r_i\in \left\{0,L_i-1\right\} \right\}$. 
If we then place $u$ and $v$ at the origin, the resulting fixed points $x=\varrho^{\infty}(u)$ and $x^{\prime}=\varrho^{\infty}(v)$ which cover $\mathbb{Z}^{d}$ will coincide at every sector except at the one where $u_{\boldsymbol{j}}\neq v_{\boldsymbol{j}}$.
One can then choose $\boldsymbol{b}_i=\boldsymbol{e}_i$ and $\alpha_i=0$ in Theorem~\ref{thm: aperiodicity Delone}, and for each $i$, $x$ and $x^{\prime}$ to be the two elements which agree on a half-space, which guarantees the aperiodicity of $\mathbb{X}_{\varrho}$.
More concretely, $x$ and $x^{\prime}$ are asymptotic, and hence proximal, along $\boldsymbol{e}_i$ for all $1\leqslant i\leqslant d$. 
\end{proof}

\begin{remark}
Obviously, one can have a lattice of periods of rank less than $d$ in higher dimensions. An example would be when $\varrho=\varrho^{ }_1\times \varrho^{ }_2$, where $\varrho^{ }_1$ is the trivial substitution $a\mapsto aa,b\mapsto bb$ and $\varrho^{ }_2$ is Thue--Morse. 
Although $\varrho^{ }_1$ is itself not primitive, the product $\varrho$ is and admits the legal blocks given in Figure~\ref{fig: TM cross per}, which generate fixed points that are $\mathbb{Z}\boldsymbol{e}_1$-periodic. 
If one requires that the shift component in $\mathcal{S}(\mathbb{X}_{\varrho})$ is $\mathbb{Z}^d$, one needs all elements of $\mathbb{X}_{\varrho}$ to be aperiodic in all cardinal directions, hence the stronger criterion in Proposition~\ref{prop: aperiodicity HD proximal}. \exend
\end{remark}

\begin{figure}[h]
\begin{center}
\includegraphics[scale=1.0]{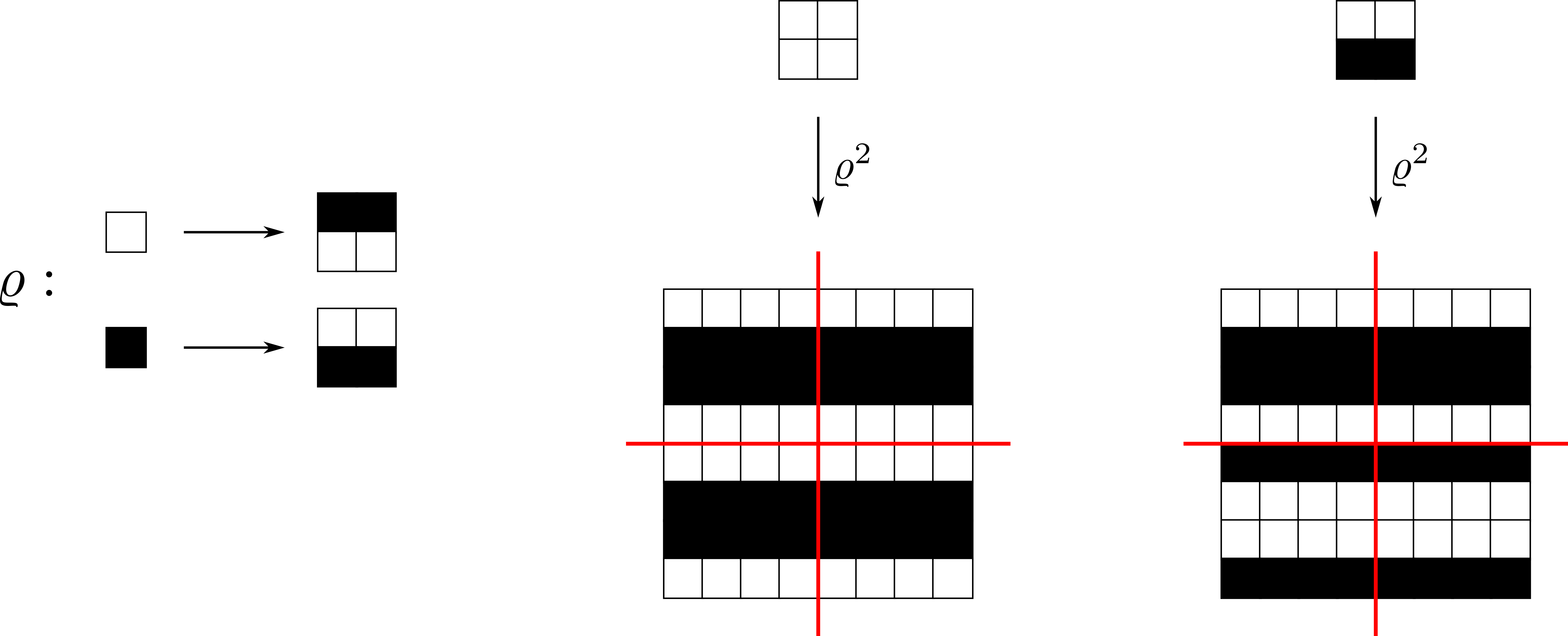}
\end{center}
\caption{The image of two distinct blocks under $\varrho$ coincide in the upper half-plane and are distinct in the lower half-plane. In the limit, these legal seeds  generate two fixed points which are neither left nor right asymptotic with respect to $\sigma_{\boldsymbol{e}_1}$. } 
\label{fig: TM cross per}
\end{figure}

The next result is the analogue of Theorem~\ref{thm: DDMP-result} for extended symmetries, which holds in any dimension.

\begin{theorem}\label{thm: G and P construction}
Given a finite group $G$ and a subgroup $P$ of the $d$-dimensional hyperoctahedral group $W_d$, there is an aperiodic, primitive, bijective $d$-dimensional substitution $\varrho$ whose shift space satisfies
\begin{align*}
\mathcal{S}(\mathbb{X}_\varrho) &\simeq \mathbb{Z}^d \times G    \\ 
\mathcal{R}(\mathbb{X}_\varrho) &\simeq (\mathbb{Z}^d \rtimes P) \times G. 
\end{align*}
\end{theorem}
\begin{proof}
We start by taking a cursory look at the proof of Theorem~3.6 in \cite{DDMP}. For a given finite group $G$, we choose a generating set $S = \{s_1,\dotsc,s_r\}$ that does not contain the identity, and build a substitution whose columns correspond to the left multiplication maps $L_{s_j}(g) = s_j\cdot g$, seen as permutations of the alphabet $\mathcal{A}=G$. These permutations generate the left Cayley embedding of $G$ in the symmetric group on $\lvert G\rvert$ elements, whose corresponding centraliser, which induces all of the letter exchanges in $\mathcal{S}(\mathbb{X}_\varrho)$, is the right Cayley embedding of $G$ generated by the maps $R_{s_j}(g) = g\cdot s_j$.

In what follows, we shall assume first that the group $G$ is non-trivial, as the case in which $G$ is trivial requires a slightly different construction. We also assume that the rectangular substitution we will construct engenders an aperiodic subshift, so that the group generated by the shifts is isomorphic to $\mathbb{Z}^d$. We delay the proof of this until later on, to avoid cluttering our construction with extraneous details.

Since $\mathcal{S}(\mathbb{X}_\varrho)$ depends only on the columns of the underlying substitution and not their relative position, we shall construct a $d$-dimensional rectangular substitution $\varrho$  with cubic support  whose columns correspond to copies of the aforementioned $L_{s_j}$, placed in adequate positions along the cube. We start with a cube $R=[0,2\lvert S\rvert + 2d + 1]^d$ of side length $2\lvert S\rvert + 2d + 2$, where the additional layer corresponding to the term $2$ will be used below to ensure aperiodicity. This cube is comprised of $N = \lvert S\rvert + d + 1$ ``shells'' or ``layers'', which are the boundaries of the inner cubes $[j,2\lvert S\rvert + 2d + 2 - j]^d$; we shall denote each of them by $\Lambda_j$, where $j$ can vary from $0$ to $N - 1$.

Fill the $i$-th inner shell $\Lambda_{N-i}$ with copies of the column $L_{s_i}$, for all $1\leqslant i\leqslant r$. This ensures that, as long as every other column is a copy of $L_{s_j}$ for some $j$ or an identity column, the symmetry group $\mathcal{S}(\mathbb{X}_\varrho)$ of the corresponding subshift will be isomorphic to $G$, because in our construction the $2^d$ corners of the point will always be identity columns. 

Now, note that $N$ is chosen large enough so that the point $\boldsymbol{p} = (0,1,\dotsc,d-1)$ lies in the outer $N-r\geqslant d$ shells and, moreover, the cube $[0,d-1]^d$ is contained in these outer shells as well. Thus, any permutation of the coordinates maps the cube $[0,d-1]^d$ to itself and, in particular, two different permutations map this point to two different points in this cube, that is, the orbit of $\boldsymbol{p}$ has $d!$ different points. Combining this with the fact that the mirroring maps send this cube to one of $2^d$ disjoint cubes (translations of $[0,d-1]^d$) in the corners of the larger cube $[0,N-1]^d$, it can be seen that $W_d$ acts freely on the orbit of the point $\boldsymbol{p}$, that is, there is a bijection between the hyperoctahedral group $W_d$ and the set $\operatorname{Orb}(\boldsymbol{p})$. 

Next, choose a fixed $s_j\in S$ that is not the identity element of $G$, so that $L_{s_j}$ is not an identity column. As $P$ is a subgroup of $W_d$, it is bijectively mapped to the set $P\cdot\boldsymbol{p} = \{g\cdot\boldsymbol{p}\::\:g\in P\}$. Place a copy of $L_{s_j}$ in each position from $P\cdot\boldsymbol{p}$, and an identity column in every other position from $\operatorname{Orb}(\boldsymbol{p})$. Fill every remaining position in the cube with identity columns. This ensures that the group of letter exchanges will remain isomorphic to $G$, and, for each matrix $A\in W_d$ associated with some element $g\in P$, the map $f_A$ given by the relation $f_A(x)_{\boldsymbol{n}} = x_{A\boldsymbol{n}}$ will be a valid extended symmetry, as a consequence of Theorem \ref{thm: extended symmetries iff}. 
 
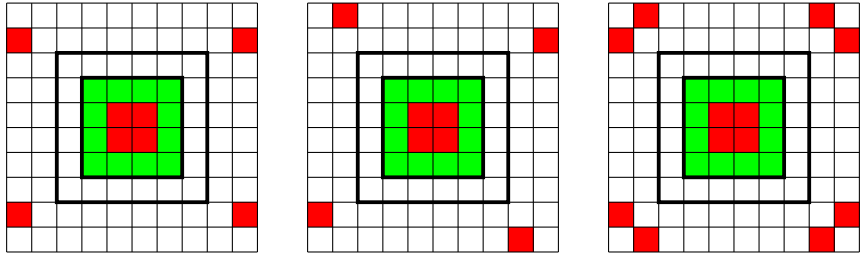
\begin{figure}[!ht]
    \centering
    \begin{tikzpicture}[scale=0.33]
        \fill[color=green] (3,3) rectangle (7,7);
        \fill[color=red] (4,4) rectangle (6,6);
        \fill[color=red] (0,1) rectangle +(1,1);
        \fill[color=red] (9,1) rectangle +(1,1);
        \fill[color=red] (0,8) rectangle +(1,1);
        \fill[color=red] (9,8) rectangle +(1,1);
        \foreach \i in {0,...,10} {
            \draw (\i,0) -- (\i,10);
            \draw (0,\i) -- (10,\i);
        }
        \draw[line width=0.5mm] (2,2) rectangle (8,8);
        \draw[line width=0.5mm] (3,3) rectangle (7,7);
        \begin{scope}[xshift=12cm]
            \fill[color=green] (3,3) rectangle (7,7);
            \fill[color=red] (4,4) rectangle (6,6);
            \fill[color=red] (0,1) rectangle +(1,1);
            \fill[color=red] (8,0) rectangle +(1,1);
            \fill[color=red] (1,9) rectangle +(1,1);
            \fill[color=red] (9,8) rectangle +(1,1);
            \foreach \i in {0,...,10} {
                \draw (\i,0) -- (\i,10);
                \draw (0,\i) -- (10,\i);
            }
            \draw[line width=0.5mm] (2,2) rectangle (8,8);
            \draw[line width=0.5mm] (3,3) rectangle (7,7);            
        \end{scope}
        \begin{scope}[xshift=24cm]
            \fill[color=green] (3,3) rectangle (7,7);
            \fill[color=red] (4,4) rectangle (6,6);
            \fill[color=red] (0,1) rectangle +(1,1);
            \fill[color=red] (9,1) rectangle +(1,1);
            \fill[color=red] (0,8) rectangle +(1,1);
            \fill[color=red] (9,8) rectangle +(1,1);
            \fill[color=red] (1,0) rectangle +(1,1);
            \fill[color=red] (1,9) rectangle +(1,1);
            \fill[color=red] (8,0) rectangle +(1,1);
            \fill[color=red] (8,9) rectangle +(1,1);
            \foreach \i in {0,...,10} {
                \draw (\i,0) -- (\i,10);
                \draw (0,\i) -- (10,\i);
            }
            \draw[line width=0.5mm] (2,2) rectangle (8,8);
            \draw[line width=0.5mm] (3,3) rectangle (7,7);
        \end{scope}

    \end{tikzpicture}
    \caption{Examples of substitutions obtained by the above construction, for the Klein $4$-group $C_2\times C_2$, the cyclic group $C_4$ and the whole $W_2=D_4$, respectively. The thicker lines mark the layer of identity columns separating the inner cube from the outer shell.}
    \label{fig:my_label}
\end{figure}

Since every other extended symmetry is a product of such an $f_A$ with some letter-exchange map that has to satisfy the conditions given by Eq.~\eqref{eq: HD sufficient} due to our construction, and $L_{s_j}$ cannot be conjugate to the identity column, the only other extended symmetries are compositions of the already extant $f_A$ with elements from $\mathcal{S}(\mathbb{X}_\varrho)$, i.e. $\mathcal{R}(\mathbb{X}_\varrho)/\mathcal{S}(\mathbb{X}_\varrho)$ has the equivalence classes of each $f_A$ as its only elements. 
As the set of all $f_A$ is an isomorphic copy of $P$ contained in $\mathcal{R}(\mathbb{X}_\varrho)$, we conclude that $\mathcal{R}(\mathbb{X}_\varrho)$ is isomorphic to the semi-direct product $\mathcal{S}(\mathbb{X}_\varrho)\rtimes P$. However, since every letter exchanges from $G$ commutes with every $f_A$ trivially, this semi-direct product may be written as $\mathcal{R}(\mathbb{X}_\varrho)\simeq (\mathbb{Z}^d\rtimes P)\times G$, as desired.

In the case where $G$ is trivial, we may choose an alphabet with at least three symbols (to ensure that $S_{\lvert\mathcal{A}\rvert}$ is non-Abelian) and repeat the construction above with a collection of columns $\varrho^{ }_0,\dotsc,\varrho^{ }_{r-1}$ that generates some subgroup of $S_{\lvert\mathcal{A}\rvert}$ with trivial centraliser (e.g. the two generators of $S_{\lvert\mathcal{A}\rvert}$ itself). The rest of the proof proceeds in the same way.

To properly conclude the proof, we need to verify that the constructed substitution generates an aperiodic shift space. We focus on the case $d>1$, as the one-dimensional case is a straightforward modification of the construction from Theorem~\ref{thm: DDMP-result}. Since our $d$-dimensional cube has at least $d+1\geqslant 3$ outer layers, we see that there is a $2\times\cdots\times 2$ cube $R_0$ contained in the outer layers that does not overlap any of the $2^d$ cubes of size $d\times\cdots\times d$ on the corners nor the inner cube of size $2\lvert S\rvert\times\cdots\times 2\lvert S\rvert$. As a consequence, this cube $R_0$ contains only identity columns. Since we have a layer $\Lambda_d$ consists only of identity columns directly enveloping the inner cube $\Lambda_{d+1}\cup\cdots\cup\Lambda_{d+\lvert S\rvert}$, the layer immediately following $\Lambda_{d}$ is comprised only of non-identity columns, which are copies of the same bijection $\pi\colon\mathcal{A}\to\mathcal{A}.$. Thus, the $2^d$ corners of the hollow cube $\Lambda_d\cup\Lambda_{d+1}$ are $2\times\cdots\times 2$ cubes $R_1,\cdots,R_{2^d}$ having exactly one non-identity column each, with this non-identity column $\tau$ being placed in every one of the $2^d$ possible positions on these cubes.

Since $\tau$ is not the identity, there must exist some $a\in\mathcal{A}$ such that $\tau(a)\ne a$. The previous discussion thus implies that there is an admissible pattern $P_a$ of size $2\times\cdots\times 2$ comprised only of copies of the symbol $a$, and $2^d+1$ other admissible patterns $P_a^{(\boldsymbol{n})}$ that differ from $P_a$ only in the position $\boldsymbol{n}\in[0,1]^d$. Using the proximality criterion from Proposition~\ref{prop: aperiodicity HD proximal}, we conclude that the subshift obtained is indeed aperiodic, as desired.
\end{proof}

\begin{remark}\label{rem: alt construction}
An alternative Cantor-type construction,  which produces the prescribed symmetry and extended symmetry groups, involves putting the non-trivial columns on the faces of $R$ and labelling all columns in the interior to be the identity. Let $G$ and $P$ be given. From Theorem~\ref{thm: DDMP-result}, there exists a substitution on $\mathcal{A}$  with $\mathcal{S}(\mathbb{X}_\varrho)=\mathbb{Z}\times G$. Let $\varrho^{ }_0,\ldots,\varrho^{ }_{r-1}$ be the non-trivial columns of $\varrho$. Pick $L$ to be large enough such that $W_d$ acts freely on the faces of $R=[0,L-1]^d$. Choose $\boldsymbol{j}_0\in R$ and consider the orbit of $\boldsymbol{j}_0$ under $P$, i.e., $\mathcal{O}_0:=P\cdot \boldsymbol{j}_0=\left\{A\cdot \boldsymbol{j}_0\mid A\in P\right\}$ where $A\cdot \boldsymbol{j}=A^{(1)}(\boldsymbol{j})$ as in Eq.~\eqref{eq: level-k g map} . Label all the columns in $\mathcal{O}_0$ with $\varrho^{ }_0$. We then expand $R$ via $Q=\text{diag}(L,\ldots,L)$ to get the $d$-dimensional cube $Q(R)$ of side length $L^2$. Consider $\mathcal{B}_1:=Q(\mathcal{O}_0)+R$, pick $\boldsymbol{j}_1\in \mathcal{B}_1$ and let $\mathcal{O}_1=P\cdot \boldsymbol{j}_1$. Relabel all columns in $\mathcal{B}_1\setminus \mathcal{O}_1$ with $\varrho^{ }_0$ and all columns in $\mathcal{O}_2$ with $\varrho^{ }_1$. One can continue this process until 
all needed column labels appear; see Figure~\ref{fig: 2D chosen groups} for a two-dimensional example. 

\begin{figure}[h!]
\subfloat[$\mathcal{O}_0$ in blue]{
\includegraphics[scale=1.2]{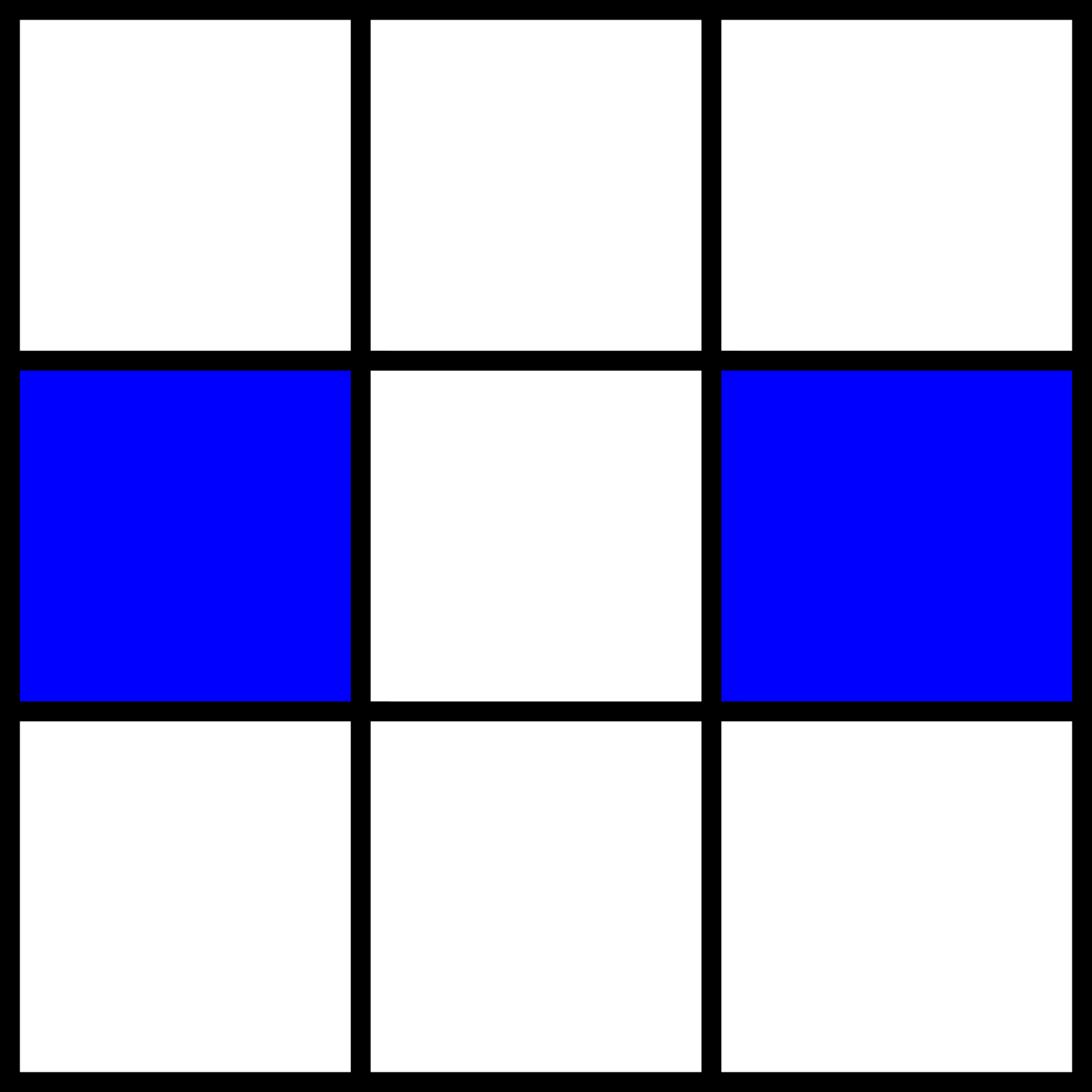}
} \quad
\subfloat[$\mathcal{B}_1\setminus\mathcal{O}_1$ in blue, $\mathcal{O}_1 $ in red ]{
\includegraphics[scale=1.2]{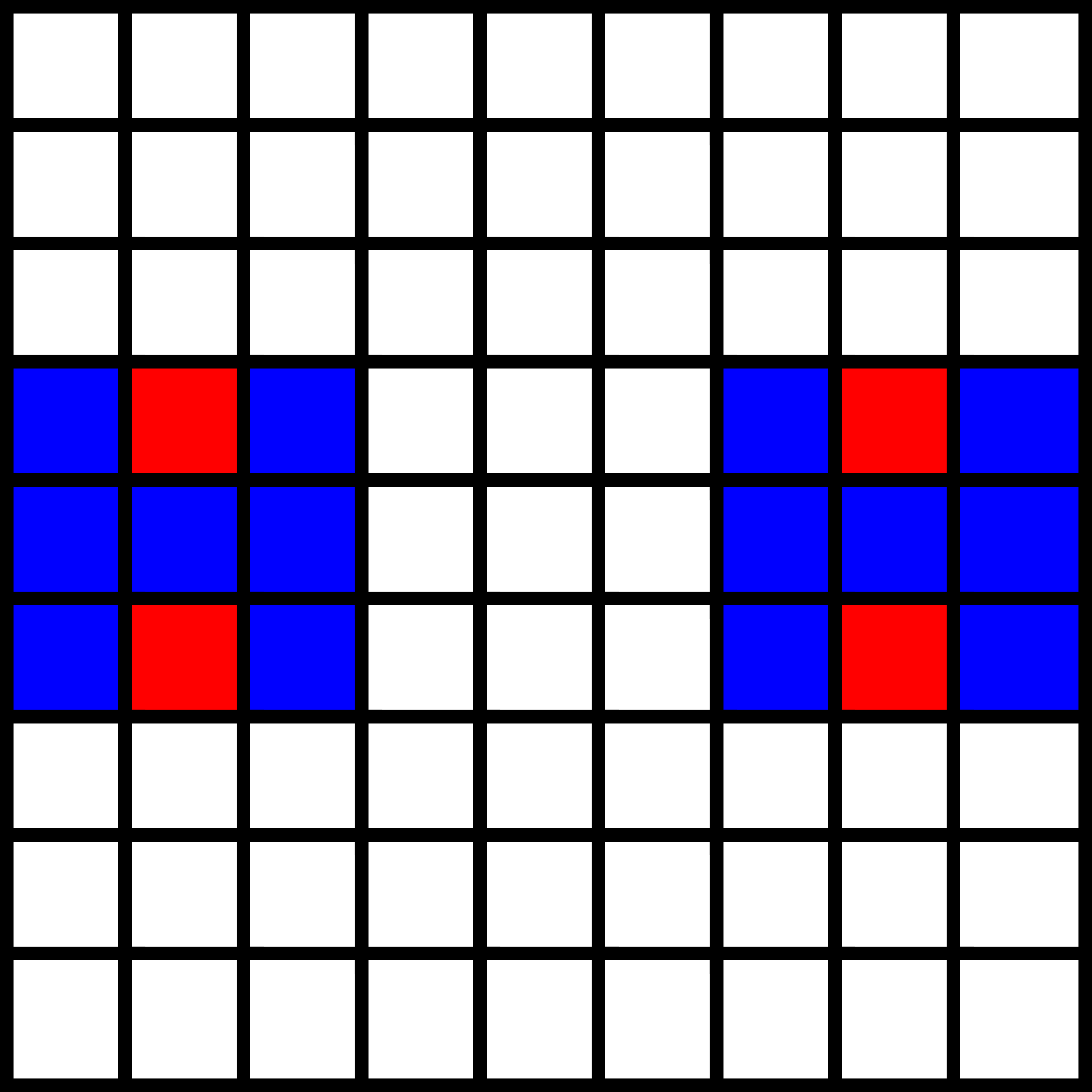}
} \quad
\subfloat[$\mathcal{B}_2\setminus \mathcal{O}_2$ in green]{
\includegraphics[scale=1.2]{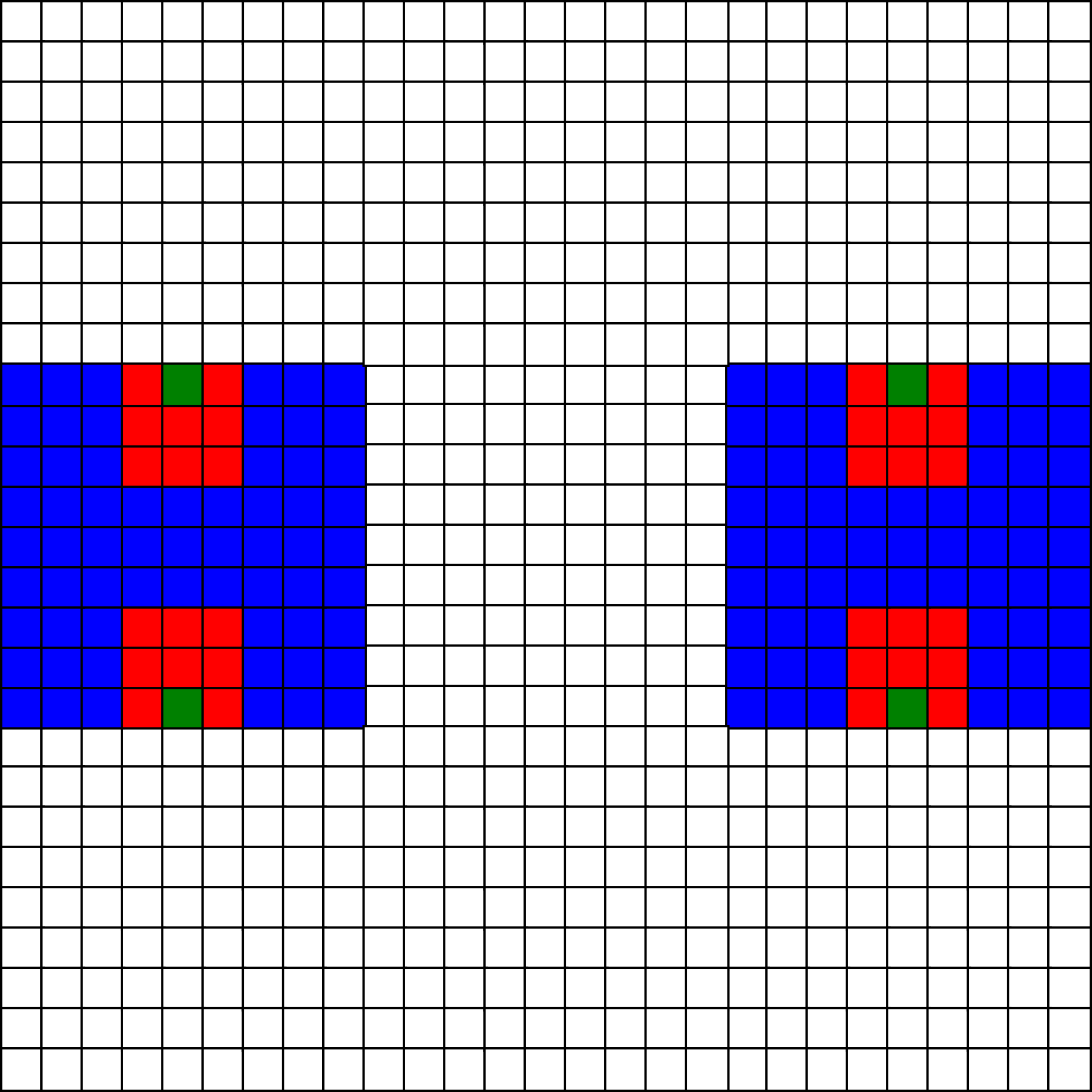}
}
\caption{An example in $\mathbb{Z}^2$ with three non-trivial columns $\varrho^{ }_0$ (blue), $\varrho^{ }_1$ (red) and $\varrho^{ }_2$ (green). Here, one has $G=\text{cent}_{S_{|\mathcal{A}|}}\left\langle\varrho^{ }_0,\varrho^{ }_1,\varrho^{ }_2\right\rangle$ and $P\simeq V_4$, where $V_4\leqslant D_4=W_2$ is the Klein-$4$ group.}\label{fig: 2D chosen groups}
\end{figure}

Note that one has 
$\varrho^{ }_{\boldsymbol{i}}=\varrho^{ }_{A^{(1)}(\boldsymbol{i})}$ for all $A\in P$ and $\boldsymbol{i}\in R=[0,L-1]^d$ by construction, which means $\pi=\text{id}$ gives rise to an element of $\mathcal{R}(\mathbb{X}_\varrho)$ for all $A\in P$ by Theorem~\ref{thm: extended symmetries iff}. No other extended symmetries can occur because all the location sets $\mathcal{B}_i$ only contain non-trivial labels and are $P$-invariant, whereas if $A\notin P$ induces an extended symmetry, one must have $\varrho^{ }_{\boldsymbol{\ell}}=\text{id}$ for some $\boldsymbol{\ell}\in \mathcal{B}_r$.

The resulting block substitution is primitive, since reordering the columns does not affect primitivity. It is also aperiodic because one has enough identity columns, and hence one can find the legal words required in Proposition~\ref{prop: aperiodicity HD proximal}. For example, in the constructed substitution in Figure~\ref{fig: 2D chosen groups}, the legal seeds 
can be derived from the $2\times2$ block consisting of all identity columns (i.e. all white squares), and another one with all columns being the identity except at exactly one corner, where it is blue. This completes the picture and one has $\mathcal{S}(\mathbb{X}_\varrho)\simeq\mathbb{Z}^d\times G$ and $\mathcal{R}(\mathbb{X}_\varrho)\simeq(\mathbb{Z}^d\rtimes P)\times G$. \exend
\end{remark}

We now turn our attention to examples where the letter-exchange map $\pi$ that generates $f\in \mathcal{R}(\mathbb{X}_{\varrho})$ is not given by the identity.
In particular, in these examples, $\pi$ does not commute with the letter-exchanges which correspond to the standard symmetries in $\mathcal{S}(\mathbb{X}_{\varrho})$.
To avoid confusion, we will use letters for our substitution and the action of the hyperoctahedral group will be given by numbers, seen as permutations of the coordinates. Mirroring along a hyperplane will be denoted by $m_i$, where $i$ is the respective coordinate. 

\begin{example}\label{ex: three dimension}
We explicitly give a substitution whose symmetry group is $\mathcal{S}(\mathbb{X}_{\varepsilon})=\mathbb{Z}^d\times C_3$ and build another $C_3$ component in $\mathcal{R}(\mathbb{X}_{\varepsilon})$, which produces reversors of order 9. With the requirement on $\mathcal{R}(\mathbb{X}_{\varepsilon})/\mathcal{S}(\mathbb{X}_{\varepsilon})$, the space has to be at least of dimension 3. 
\begin{align*}
    \varepsilon^{ }_0&=(a\,d\,g)(b\,e\,h)(c\,f\,i)  & \varepsilon^{ }_2&=(a\,b\,c)(d\,e\,f)(g\,h\,i)  &\varepsilon^{ }_5=\textnormal{id} \\
    \varepsilon^{ }_1&=(a\,g\,d)(b\,h\,e)(c\,i\,f) &\varepsilon^{ }_3&=(b\,c\,d)(e\,f\,g)(h\,i\,a)    \\
    & &\varepsilon^{ }_4&=(c\,d\,e)(f\,g\,h)(i\,a\,b) 
\end{align*}
Here one has $\mathcal{S}(\mathbb{X}_\varepsilon)= \mathbb{Z}^3 \times C_3 $, which is generated by $(a\,d\,g)(b\,e\,h)(c\,f\,i)$. Depending on the positioning of the columns, $\mathcal{R}(\mathbb{X}_\varepsilon)$ can either be $\mathbb{Z}^3 \rtimes C_9$, $\mathbb{Z}^3 \rtimes C_3 \times C_3$ or $ \mathbb{Z}^3 \times C_3 $. The group $\mathbb{Z}^3 \rtimes C_3 \times C_3$ can be realised using the construction from Theorem~\ref{thm: G and P construction}. On the other hand, $\mathbb{Z}^3 \times C_3 $ is obtained if one orbit of maximal size is labelled with just one non-identity $\varepsilon^{ }_i$ once, and the rest with $\varepsilon^{ }_0$. 

Note that $\pi=(a\,b\,c\,d\,e\,f\,g\,h\,i)$ sends $\varepsilon^{ }_2 \rightarrow \varepsilon^{ }_3 \rightarrow \varepsilon^{ }_4 \rightarrow \varepsilon^{ }_2 $ and $\varepsilon^{ }_0 \rightarrow \varepsilon^{ }_0, \; \varepsilon^{ }_1 \rightarrow \varepsilon^{ }_1$.
Taking the cube of $(a\,b\,c\,d\,e\,f\,g\,h\,i)$ gives $(a\,d\,g)(b\,e\,h)(c\,f\,i)\in \text{cent}_{S_9}(G^{(1)})$, where $G^{(1)}$ is the group generated by the columns. This is consistent with the bounds calculated in Proposition~\ref{prop:Rorder}. We will illustrate the positioning of a few elements following the construction in Theorem~\ref{thm: G and P construction}. We look at a position that has the maximum orbit size under $W_3$, for example $(0,1,2)\in R$. The orbit under $C_3$ is $(0,1,2),(1,2,0),(2,0,1)$, which is obtained by cyclically permuting the coordinates. We place $\varepsilon^{ }_2$ at position (0,1,2), $\varepsilon^{ }_3$ at position (1,2,0) and $\varepsilon^{ }_4$ at position (2,0,1).  Since $\varepsilon^{ }_0,\varepsilon^{ }_1\in \text{cent}_{S_9}(G)$, we will position them each along a different orbit. All remaining positions will be filled with the identity to ensure that we cannot have additional symmetries. We use Proposition~ \ref{prop: aperiodicity HD proximal} to ensure aperiodicity. It is easy to see that one gets the required patches by choosing the $2\times2\times 2$ cube in the upper right corner from the first and second slices and the other one from the second and third. For this configuration, one has $\mathcal{R}(\mathbb{X}_{\varepsilon})=\mathbb{Z}^3\rtimes C_9$. \exend

\end{example}
\begin{figure}[H]
\includegraphics[scale=0.15]{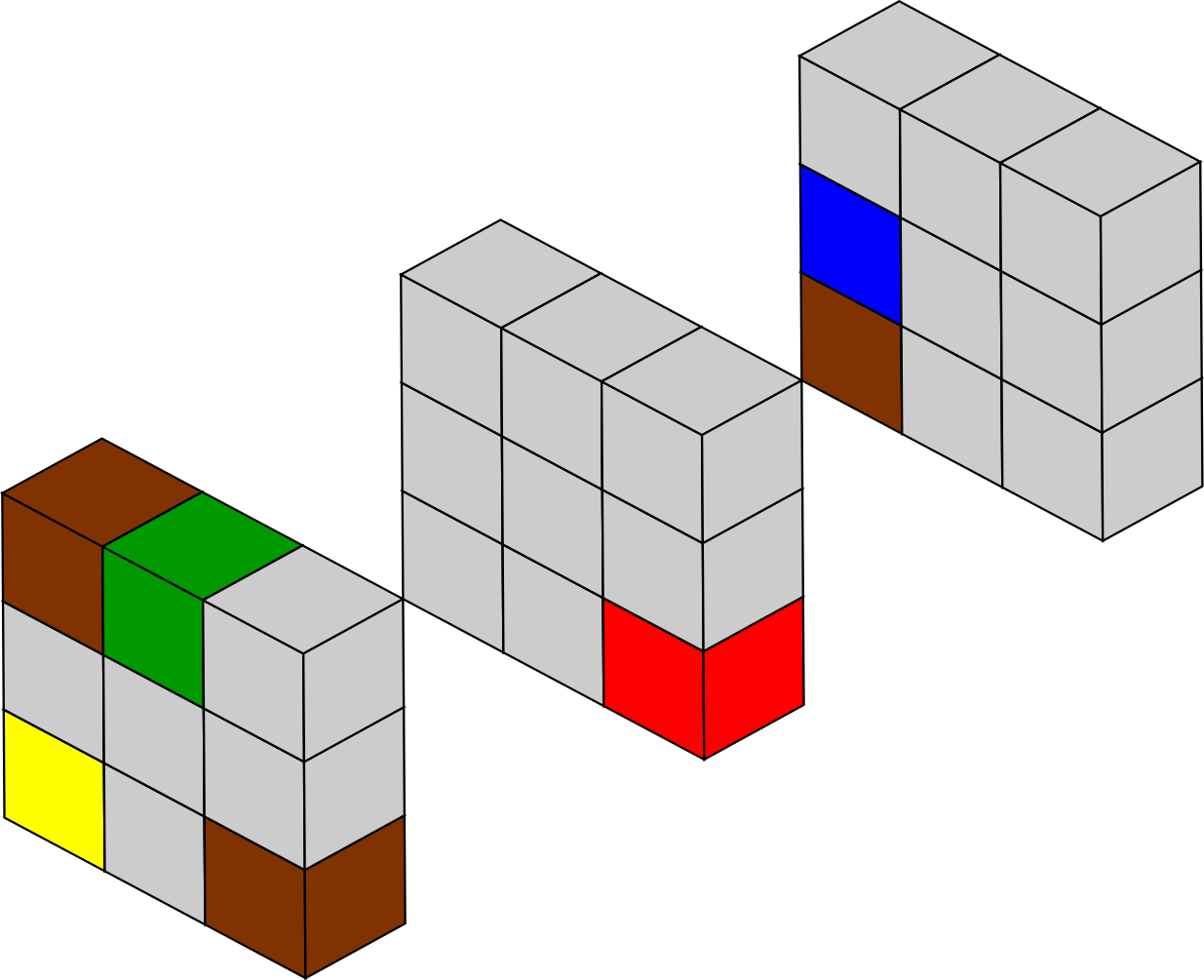}
\caption{The gray cubes are filled with $\varepsilon^{ }_5$ (the identity). Yellow and brown can be filled by either $\varepsilon^{ }_0, \varepsilon^{ }_1$, respectively. Lastly, $\varepsilon^{ }_2$ is blue, $\varepsilon^{ }_3$ is green and $\varepsilon^{ }_4$ is red, where one has the obvious freedom in choosing the colours due to the $C_3$-symmetry.}
\label{fig:positioning_col_1}
\end{figure}

\begin{remark} As a generalisation of Example~\ref{ex: three dimension}, for any given cyclic groups $C_n$ and $C_k$, we can construct a substitution $\varrho$ in $\mathbb{Z}^{n}$, such that $\mathbb{X}_\varrho$ has the symmetry group $\mathbb{Z}^{n} \times C_k$ and its extended symmetry group is given by $(\mathbb{Z}^{n} \times C_k) \rtimes C_n$. More precisely, since the extended symmetry group contains an element of order $nk$, $\mathcal{R}(\mathbb{X}_{\varrho})=\mathbb{Z}^{n}\rtimes C_{nk}$. 
The substitution can be realised by the following columns 
\begin{align*}
\varepsilon^{ }_0&=(a_1 \, a_{k+1} \, \cdots \, a_{(n-1)k+1}) \cdots (a_k \,  \cdots \, a_{nk}) \\
\varepsilon^{ }_i&=(a_i \, a_{i+1}\, \cdots \, a_{k-1+i}) (a_{k+i} \, a_{k+i+1} \,  \cdots \, a_{2k+i-1}) \cdots (a_{(n-1)k+i} \, a_{(n-1)k+i+1} \, \cdots \, a_{nk-1+i})  \\
\varepsilon^{ }_{n+1}&= \text{id}
\end{align*}
where $i$ runs from 1 to $n$, where the values are seen modulo $nk$.

From the columns $\varepsilon^{ }_i$ with $i\neq 0$ we can see that the centraliser can only be the permutation of the cycles limiting the centraliser to $S_k$, while $\varepsilon_0$ limits it further to be $C_k$, since this copy of $S_k$ operates on the cycles independently and the centraliser of a cycle is just the cycle itself. The extended symmetry is realised by the permutation $(a_1 \cdots a_{kn})$ which maps $\varepsilon^{ }_i$ to $\varepsilon^{ }_{i+1}$. Its orbit is determined by the action of $C_n \leqslant W_n$ on the positioning of the columns.  \exend
\end{remark}

In the next example we illustrate how important it is to choose compatible structures for the letter-exchange map and the corresponding action in $W_d$.

\begin{example}\label{ex: S4 example}
 We look at a four-letter alphabet with the following columns in Eq.~\eqref{eq: columns S4} which generate $S_4$ as a subgroup of $S_4$, thus implying that the shift space to have a trivial centraliser. We plan to have $S_3\simeq \mathcal{R}(\mathbb{X}_{\varepsilon})/\mathcal{S}(\mathbb{X}_{\varepsilon})$, so we place the columns in a three-dimensional cube.

\begin{align}\label{eq: columns S4}
    \varepsilon^{ }_0&= \textnormal{id}      &\varepsilon^{ }_1&=(a\,b\,c\,d)   &\varepsilon^{ }_4&=(a\,c\,d\,b) \nonumber\\
                        & &\varepsilon^{ }_2&=(a\,b\,d\,c)  &\varepsilon^{ }_5&=(a\,d\,b\,c)\\
                        & &\varepsilon^{ }_3&=(a\,c\,b\,d)   &\varepsilon^{ }_6&=(a\,d\,c\,b)  \nonumber
\end{align}
The symmetry group is trivial since the columns generate $S_4$. Conjugation with $\tau=(c\,d)$ maps $\varepsilon^{ }_1$ to $\varepsilon^{ }_2$, just as any $\tau \kappa$, with $\kappa\in  \text{cent}_{S_4}(\varepsilon^{ }_2)$.

\begin{figure}[H]
\includegraphics[scale=0.15]{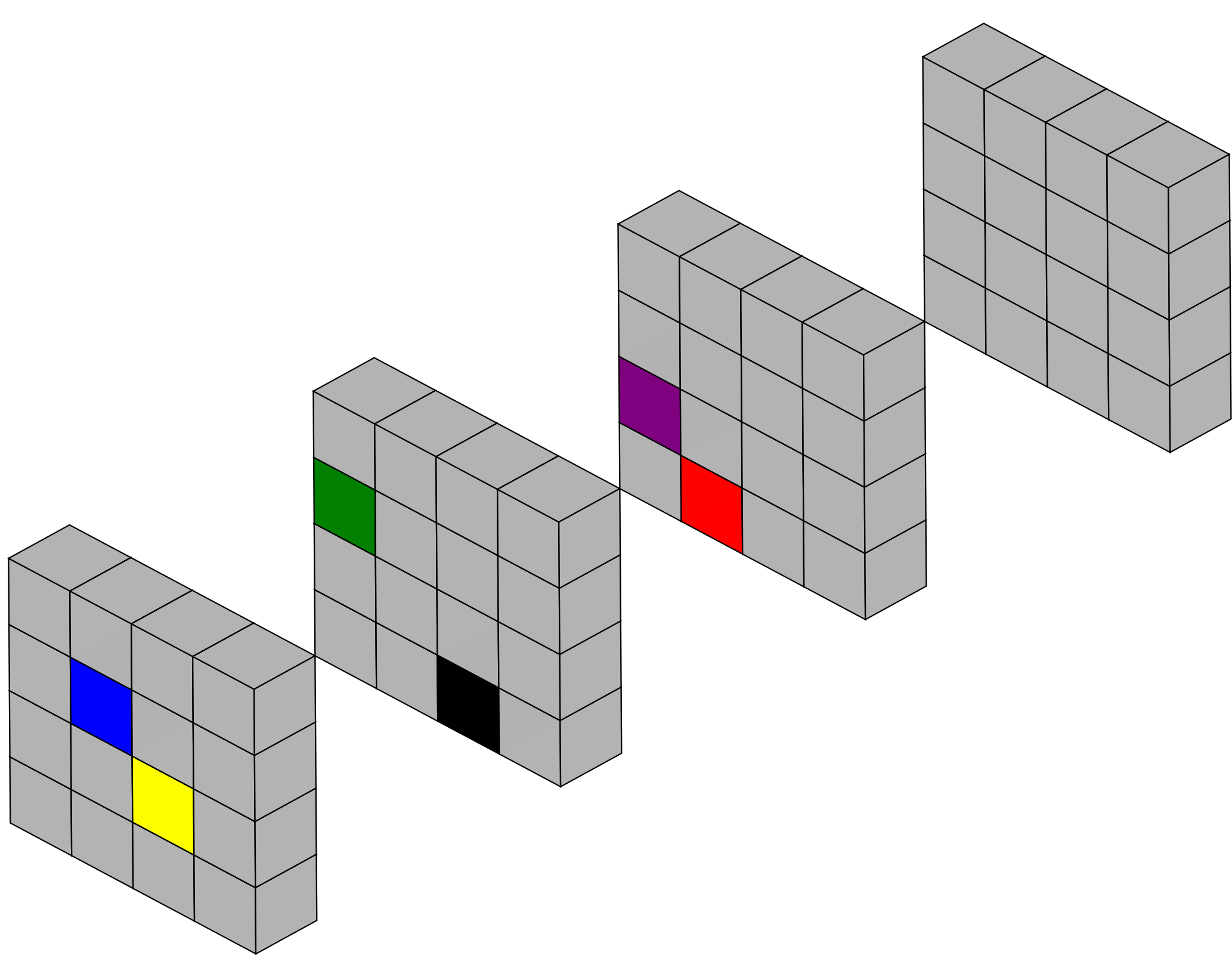}
\caption{The columns assigned to the colors are as follows: $\varepsilon^{ }_1$ (blue),\\ $\varepsilon^{ }_2$ (yellow), $\varepsilon^{ }_3$ (green) $\varepsilon^{ }_4$ (purple), $\varepsilon^{ }_5$ (black) and $\varepsilon^{ }_6$ (red).}
\label{fig:positioning_cube}
\end{figure}

Here $C_3 \rtimes C_2 \simeq S_3$ is realised by $(b\,c\,d)(0\,1\,2)$ and $(c\,d)(0\,1)$. 
The transposition $(c\,d)$ cannot be realised in $W_d$ by mirroring along an axis in the cube since that is not consistent with the interaction between $(b\,c\,d)$ and $(c\,d)$. This can be easily be seen by looking at mirroring along all hyperplanes.  
\begin{center}
\begin{tikzcd}
(a\,b\,c\,d) (2,1,0) \arrow[rr,shift left,"(b\,c\,d)"] \arrow[rr, shift right,"(0\,1\,2)"'] \arrow[dd, shift left, "m_{012}"] \arrow[dd, shift right, "(c\,d)"']&& (a\,c\,d\,b) (0,2,1) \arrow[dd, shift left, "m_{012}"] \arrow[dd, shift right, "(c\,d)"']\\
\\
(a\,b\,d\,c)(1,2,3) \arrow[rr,shift left,"(b\,c\,d)"] \arrow[rr, shift right,"(0\,1\,2)"'] && (a\,d\,c\,b)/(a\,c\,b\,d) (3,1,2) 
\end{tikzcd}     
\end{center}
We see that the diagram does not commute, thus there is no way to assign a single column to the vertex (3,\,1,\,2). One can do this for all axes, which rules out the $C_2^3$ component in $W_3$, thus yielding $\mathcal{R}(\mathbb{X}_{\varrho})=\mathbb{Z}^d\rtimes S_3$. 
\exend
\end{example}

\begin{remark}
One can also ask whether, starting with a group $G$, one can build the centraliser $\mathcal{S}(\mathbb{X}_{\varrho})$ and  normaliser $\mathcal{R}(\mathbb{X}_{\varrho})$ organically from $G$, under a suitable embedding of $G$. Consider the Cayley embedding $G\hookrightarrow S_{|G|}$  as in Example~\ref{ex:quaternions}. We know that $\text{cent}_{S_{|G|}}(G)\simeq G$ and  $\text{norm}_{S_{|G|}}(G)\simeq G\rtimes \text{Aut}(G)$; see \cite{Seh}. Since
the automorphisms of $G$ are given by conjugation in $S_{|G|}$, they define letter-exchange maps which are compatible with reversors in $\mathcal{R}(\mathbb{X}_{\varrho})$.
By choosing the dimension appropriately, one can  construct a substitution $\varrho$ on $\mathcal{A}=G$ such that the extended symmetry group is given by 
\[
\mathcal{R}(\mathbb{X}_\varrho)= \big(\mathbb{Z}^{d(G)}\times G \big)\rtimes \text{Aut}(G),
\]
where we choose $d(G)$ such that  $\text{Aut}(G)\leqslant W_{d(G)}$. This can always be done for $d(G)=|G|$, but depending on $\text{Aut}(G)$, a smaller dimension is possible. Let $\pi\in \text{Aut}(G)$ and let $A_{\pi}\in W_d$. 
The construction from the proof of Theorem \ref{thm: G and P construction} can be applied. Here, the orbits of $A_{\pi}$  will not be filled with the same element, but with columns that are determined by $\pi$, i.e., $\varrho^{ }_{A_{\pi}(\boldsymbol{i})}=\pi\circ\varrho^{ }_{\boldsymbol{i}}\circ\pi^{-1}$, where $\pi$ is seen as an element of $S_{|G|}$.  \exend
\end{remark}

These series of examples with more complicated structure can be generalised for arbitrary groups $G$ and $P$. Here, we have the following version  of Theorem~\ref{thm: G and P construction} where the letter exchange map is no longer $\pi=\text{id}$, which we build from a specific set of columns.

\begin{theorem}\label{thm: G and P non id}
Let $H,P$ be arbitrary finite groups. Then
for all $\ell\geqslant c(P)$, where $c(P)$ is a constant which depends only on the group $P$, there is a shift space $\mathbb{X}_{\varrho}$ originating from an aperiodic, primitive and bijective substitution $\varrho$ such that 
\begin{align*}
    \mathcal{S}(\mathbbm{X_\varrho})&=\mathbbm{Z}^\ell \times H\\
    \mathcal{R}(\mathbbm{X_\varrho})&=(\mathbbm{Z}^\ell \times H ) \rtimes P.
\end{align*} 
\end{theorem}

\begin{proof} 
The proof will be divided into two parts, beginning with a manual for the construction of the substitution and a second part where we verify the claims made in the construction and check if the subshift has the desired properties.

\begin{itemize}
    \item We first turn our attention to the construction of $P$ which later is supposed to be isomorphic to $\mathcal{R}(\mathbbm{X_\varrho}) / \mathcal{S}(\mathbbm{X_\varrho})$. For that purpose we embed $P \hookrightarrow S_\ell$ which is certainly possible for some $\ell$. It is clear that there is a minimal $c(P)\in\mathbb{N}$ for which this embedding is possible, and that every $\ell\geqslant c(P)$ gives a valid embedding as well. This means the choice of $\ell$ has a lower bound, but can be increased arbitrarily. This chosen $\ell$ determines the dimension of the space $\mathbb{Z}^{\ell}$ where the subshift is constructed. Let us now fix a suitable $\ell$, excluding $\ell=2,3,6$ since we use want to use $\text{Aut}_{S_\ell}(S_\ell)=\text{Inn}_{S_\ell}(S_\ell)\simeq S_\ell$ which does not hold for these values of $\ell$; see \cite{Seg}. 
    \item Next, we look for suitable columns for our substitution. Choose the set $T=\{ \epsilon^{}_1, \cdots \epsilon^{}_k \}$ of all transpositions in $S_{\ell}$, together with the identity column as the set of columns. $T$ generates $S_\ell$ and the action of $S_\ell$ (viewed as the automorphism group) acts faithfully on $T$. From this, we get that $P \subset S_\ell \simeq \text{Inn}_{S_{\ell}}(S_{\ell}) \subset \text{norm}_{S_\ell}  (\{ \epsilon^{}_1, \cdots, \epsilon^{}_k \}  ) $. 
    This is enough for now, since $P \subset \text{norm}_{S_\ell}(\{ \epsilon^{}_1, \cdots, \epsilon^{}_k \})$ and we can exclude the surplus later.  
    \item Now, we compute the centraliser of the column group. In our current construction the centraliser is trivial, which is why we need to modify our columns. We do this by extending our alphabet $\{a, \cdots, \ell\}$ to $\{a^{}_1, \cdots, a^{}_{|H|},b^{}_1, \cdots, b^{}_{|H|}, \cdots, \ell^{}_1, \cdots, \ell^{}_{|H|} \}$. We simply duplicate the cycles in each column: The permutations of the columns are mapped by $\rho  \rightarrow \rho ' $ sending $\epsilon^{}_i=(x \, y) \mapsto \varepsilon^{}_i=(x^{}_1 \, y_1^{}) \cdots ( x^{}_{|H|} \,y^{}_{|H|})$. 
    \item We embed $G \hookrightarrow S_{|H|}$ with the usual Cayley embedding. This group is only acting on the indices of the letters in the new alphabet. The action on the indices is applied to every $\{a, \ldots, \ell\}$, giving the final set of columns $\{ \eta^{}_1, \ldots, \eta^{}_m \}$ added to the substitution $\rho'$ giving a new substitution $\varrho$. 
    \item The Cayley embedding guarantees that $\text{cent}_{S_{|H|}}(G_\varrho) \simeq H$, where $G_\varrho$ is the column group of $\varrho$. We can decrease the size of $\mathcal{R}(\mathbbm{X_\varrho}) / \mathcal{S}(\mathbbm{X_\varrho})$ with the same arguments as in Theorem~\ref{thm: G and P construction}. This way we achieve a group $\mathcal{R}(\mathbbm{X_\varrho})/\mathcal{S}(\mathbbm{X_\varrho}) \simeq P$ where the letter exchange component $\pi$ of the extended symmetries are not in $\text{cent}_{S_{|H|}}(G_\varrho)$. 
\end{itemize}

Aperiodicity of $\mathbb{X}_\varrho$ can be easily obtained via proximal pairs. Regarding primitivity, it is sufficient to check the transitivity of $G_\varrho$ and use Proposition \ref{prop:group_primitivity}. For any pair $(x^{ }_j,y^{ }_k)$ of letters with indices chosen from the alphabet we need to find a $g \in G_\varrho$ such that $gx^{}_j=y^{}_k$. Note that the permutation $(x^{}_1 \, y^{}_1) \cdots (x^{}_j \, y^{}_j) \cdots (x^{}_{|H|} \, y^{}_{|H|})\in G_\varrho$ and maps $x^{}_j$ to $y^{}_j$. Now we need to map $y^{}_j$ to $y^{}_k$, which is an action solely on the indices. The mapping on the indices can be realized by the right embedding copy of $H$ in $S_{|H|}$ and thus by an element composed of the columns $\{ \eta^{}_1, \cdots, \eta^{}_m \}$.

Let us prove that the centraliser is indeed isomorphic to $G$. The centraliser of $G_{ \{ \varepsilon^{}_1, \cdots \varepsilon^{}_k \} }$ can only contain elements that are pure index permutations, since those columns generate $S_\ell$. Since the structure of the cycles in each column are independent of the index, any index permutation is an element of $\text{cent}_{S_\ell}(G_{ \{ \varepsilon^{}_1, \cdots \varepsilon^{}_k \} })=S_\ell$. 

 We continue by determining $\text{cent}_{S_{\ell|H|}}(G_{ \{ \eta^{}_1, \cdots, \eta^{}_m \} }) \bigcap S_\ell$. The group $S_\ell$ are the pure index switches and since $\eta^{}_1, \cdots, \eta^{}_m$ are the columns generated by the Cayley embedding of $H$ into $S_\ell$ their centraliser is isomorphic to $H$. 

The following rule lifts an automorphism $h'$ on $G_\rho$ to $h$ on $G_\varrho$ . 
\begin{align*}
    h(\epsilon^{}_i)= h'(\epsilon)^{}_i
\end{align*}
Thus  $S_{|H|} \leqslant \text{Aut}_{S_{\ell|H|}}(G_{ \{ \varepsilon^{}_1, \cdots, \varepsilon^{}_k \} })$. It is sufficient to prove that the automorphism group did not decrease in size by the addition of the columns $(\eta^{}_1, \cdots, \eta^{}_m)$. Then we can use the geometric placement of the columns in Theorem~\ref{thm: G and P construction} in the substitution to exclude any unwanted $W_d$-component. Any lifted automorphism $h$ still only maps the letters and fixes the indices. Since the cycles in any  $\eta^{}_1, \cdots, \eta^{}_m$ contain only the same letter with different indices and the index structure is independent of the letter, every $h$ is in $\text{cent}_{S_{|H|\ell}}(G_{ \{ \eta^{}_1, \cdots, \eta^{}_m \} })$ and surely legal. Thus it is an automorphism on the whole of $G_\varrho$.
\end{proof}

\begin{remark}
Theorems~\ref{thm: G and P construction} and \ref{thm: G and P non id} fall under realisation theorems for shift spaces. The most general current result along this vein known to the authors is that of Cortez and Petite, which states that every countable group $G$ can be realised as a subgroup $G\leqslant \mathcal{R}(\mathbb{X},\varGamma)$, where  $\mathcal{R}(\mathbb{X},\varGamma)$ is the normaliser of the action of a free abelian group $\varGamma$ on an aperiodic minimal Cantor space $\mathbb{X}$; see \cite{CP}.  \exend 
\end{remark}

\section{Concluding remarks}

While the higher-dimensional criteria in Theorems~\ref{thm: extended symmetries iff} and \ref{thm: exclude HD symmetries}, which confirm or rule out the existence of extended symmetries, are rather general, it remains unclear how to find a way to extend this to a larger (possibly all) class of systems, with no constraints on the geometry of the supertiles. 
This is related to a question of determining whether, given a substitution in $\mathbb{Z}^d$ (or $\mathbb{R}^d$), one can come up with an algorithm which decides whether there is a simpler substitution  which generates the same or a topologically conjugate hull, which is easier to investigate. This is exactly the case for the two-dimensional Thue--Morse substitution  in Figure~\ref{fig:alternate_thuemorse}. Such an issue is non-trivial both in the tiling and the subshift context; see \cite{CD,DL,HRS}.

Note that the letter-exchange map $\pi\in S_{|\mathcal{A}|}$ in Theorem~\ref{thm: extended symmetries iff} always induces a conjugacy between columns whenever it generates a valid reversor. 
 It would be interesting to know whether outer automorphisms in this case can yield valid reversors for a bijective substitution subshift in $\mathbb{Z}^d$, for example for those whose geometries are not covered by Theorems~\ref{thm: extended symmetries iff} and \ref{thm: exclude HD symmetries}.  
For instance, $\text{Aut}(S_6)$ contains elements which are not realised by conjugation. 

Another natural question would be to determine other possibilities for $\mathcal{S}(\mathbb{X}_{\varrho})$ and $\mathcal{R}(\mathbb{X}_{\varrho})$ outside the class of bijective, constant-length substitutions.
Here, the higher-dimensional generalisations of  the Rudin--Shapiro substitution would be good candidates; see \cite{Frank}. There are also substitutive planar tilings with $|\mathcal{R}(\mathbb{X})/\mathcal{S}(\mathbb{X})|=D_6$, which arises from the hexagonal symmetry satisfied by the underlying tiling. For these classes, and in the examples treated above, the simple geometry of the tiles introduces a form of rigidity which leads to $\mathcal{R}(\mathbb{X})$ being  a finite extension of $\mathcal{S}(\mathbb{X})$; see \cite[Sec.~5]{BRY} for the notion of hypercubic shifts. There are substitution tilings whose expansive maps $Q$ are no longer diagonal matrices, and whose supertiles have fractal boundaries; compare \cite[Ex.~12]{Frank3}, which allows more freedom in terms of admissible elements of $\text{GL}(d,\mathbb{Z})$ which generate reversors. This raises the following question:

\begin{question}
What is the weakest condition on the shift space/tiling dynamical system $\mathbb{X}$ which guarantees $\left[\mathcal{R}(\mathbb{X}):\mathcal{S}(\mathbb{X})\right]<\infty$?  
\end{question}

This is always true in one dimension regardless of complexity, since either the subshift is reversible or not, but is non-trivial in higher dimensions because $|\text{GL}(d,\mathbb{Z})|=\infty$ for $d>1$, so infinite extensions are possible; see \cite{BBHLN}. 
We suspect that this is connected to the notions of linear repetitivity, finite local complexity, and rotational complexity; compare \cite[Cor.~4]{BRY} and \cite{HRS}. 
 For inflation systems,
the compatibility condition $[A,Q]=0$ in Theorem~\ref{thm: extended symmetries iff} might also be necessary in general when the maximal equicontinuous factor (MEF) has an explicit form. 

\section{Acknowledgements}

The authors would like to thank Michael Baake for fruitful discussions and for valuable comments on the manuscript. AB is grateful to ANID (formerly CONICYT) for the financial support received under the Doctoral Fellowship ANID-PFCHA/Doctorado Nacional/2017-21171061. NM would like to acknowledge the support of the German Research Foundation (DFG) through the CRC 1283.

\bibliographystyle{alpha}

\end{document}